\documentclass[a4paper,oneside,11pt]{scrartcl}
\input{mypaper.sty}
\usepackage{mathrsfs}
\usetikzlibrary{decorations.pathmorphing,patterns}

\def\adm {\textnormal{adm}}
\def\nadm {\textnormal{nonadm}}
\def\nmin   {n_{\min}^\H}
\def\nminH   {n_{\min}^{\H^2}}
\def\H      {\mathcal{H}}

\def\deta {\textnormal{det}\,}
\renewcommand{\L}[1]{{\op{L}(#1)}}

\begin{document}
\title{Kernel-independent adaptive construction of $\H^2$-matrix approximations}
\author{M.\ Bauer, M.\ Bebendorf, and B.\ Feist\footnote{Faculty of Mathematics, Physics and Computer Science, Universit\"at Bayreuth, 95447 Bayreuth, Germany}}
\date{\today}
\maketitle

\begin{abstract}
A method for the kernel-independent construction of $\H^2$-matrix approximations to non-local operators is proposed. 
Special attention is paid to the adaptive construction of nested bases. As a side result, new error
estimates for adaptive cross approximation~(ACA) are presented which have implications on the pivoting strategy of ACA.
\end{abstract}

\textbf{Keywords:} non-local operators, adaptive cross approximation, $\mathcal{H}^2$-matrices, interpolation

\section{Introduction}
The fast multipole method introduced by Greengard and Rokhlin (see \cite{MR86k:65120,MR88k:82007}) has become a very popular method for the
efficient evaluation of long-range potentials and forces in the $n$-body problem. In a SIAM News article~\cite{ci00} it has been
named to be one of the top 10 algorithms of the 20th century. While in the initial publications two-dimensional
electrostatic problems were investigated, later publications~\cite{MR99c:65012,MR2000h:65178} have improved the method such that
three-dimensional electrostatic problems and also problems with more general physical background can be treated efficiently.
All these variants rely on explicit kernel expansions, which on the one hand allows to tailor the expansion
tightly to the respective problem, but on the other hand requires
its own analytic apparatus including a-priori error estimates for each kernel. In order to overcome this
technical difficulty, kernel-independent generalizations~\cite{MR2054351} were introduced. While the latter keep the analytic point of
view, $\H$- and $\H^2$-matrices (see~\cite{MR2000c:65039,MR2001i:65053,HAKHSA00}) generalize the method as much as possible by an algebraic
perspective. In addition to the $n$-body problem, the latter methods can be applied to general elliptic boundary value
problems either in its differential or its integral representation; see \cite{Bebendorf:2008,WH15}. Furthermore,
approximate replacements of usual matrix operations such as addition, multiplication, and inversion can be carried out
with logarithmic-linear complexity, which allows to construct preconditioners in a fairly automatic way.

Nevertheless, $\H^2$-matrix approximations cannot be constructed without taking into account the analytic background.
For instance, the construction of suitable cluster bases is a crucial task.
In order to guarantee as much universality of the method as possible, polynomial spaces
are frequently used; see \cite{BOLOME02}. While this choice is quite convenient due to special properties of polynomials,
it is usually not the most efficient approach. To see why, keep in mind that the three-dimensional approach based on
spherical harmonics~\cite{MR2000h:65178} requires $k=\ord{p^2}$ terms in a truncated expansion with precision of order~$p$,
while the number of polynomial terms for the same order of precision requires $k=\ord{p^3}$ terms.

The number of terms $k$ required to achieve a prescribed accuracy is crucial for the overall efficiency of the method.
In addition to its dependence on the kernel, this number also depends on the underlying geometry (local patches of the geometry may have a smaller dimension).
Additionally, a-priori error estimates usually lead to an overestimation of~$k$. It is therefore helpful to
find $k$ in an automatic way, i.e.\ by an adaptive procedure. Such a method has been introduced by one of the authors.
The \emph{adaptive cross approximation}~(ACA)~\cite{MR2001j:65022} computes
low-rank approximations of suitable sub-blocks using only few of the original matrix entries. 
From the algorithmic point of view this procedure is similar to a rank-revealing LU factorization. Therefore, it is kernel 
independent. In addition to that, it provably achieves asymptotic optimal convergence rates.

The aim of this article is to generalize the adaptive cross approximation method,
which was introduced for $\mathcal{H}$-matrices, to the kernel-independent construction of $\H^2$-matrices for
matrices $A\in\R^{M\times N}$ with entries of the form
\begin{equation}\label{eq:matA}
a_{ij}=\int_\Omega\int_\Omega K(x,y)\fie_i(x)\psi_j(y)\ud y\ud x,\quad i=1,\dots,M,\; j=1,\dots,N.
\end{equation}
Here, $\fie_i$ and $\psi_j$ denote locally supported ansatz and test functions. The kernel function~$K$ is of the type
\begin{equation}\label{eq:matA1}
  K(x,y)=\xi(x)\,\zeta(y) \, f(x,y)
\end{equation}
with a singular function~$f(x,y)=|x-y|^{-\alpha}$ and functions~$\xi$ and~$\zeta$ each depending on only one of the variables~$x$ and~$y$.  Such matrices result, for instance, from a Galerkin
discretization of integral operators. In particular, this
includes the single layer potential operator
$K(x,y)=|x-y|^{-1}$ and the double layer potential operator of the Laplacian in $\R^3$ for which $K(x,y)=\frac{(x-y)\cdot n_y}{|x-y|^3}=\frac{x\cdot n_y}{|x-y|^3}-\frac{y\cdot n_y}{|x-y|^3}$.
Note that collocation methods and Nystrom methods can also be included by formally choosing
$\fie_i=\delta_{x_i}$ or $\psi_j=\delta_{x_j}$, where $\delta_x$ denotes the Dirac distribution centered at~$x$.
In contrast to $\H$-matrices for which the method is applied to blocks, in the case of $\H^2$-matrices
cluster bases have to be constructed. If this is to be done adaptively, special properties of the kernel have to
be exploited, in order to be able to guarantee that the error is controlled also outside of the cluster. Our approach relies on the
harmonicity of the singular part~$f$ of the kernel function~$K$. This article
also presents a-priori error estimates which are based on interpolation by radial basis functions.
The advantage of these new results is that they pave the way to a new pivoting strategy of ACA.
While results based on polynomial interpolation error estimates require that the pivots are chosen such that
unisolvency of the polynomial interpolation problem is guaranteed, the new estimates show that only the fill distance
of pivoting points is crucial for the convergence of ACA.

The article is organized as follows. In the next Sect.~\ref{sec:hiqr} we construct interpolants~$s_k$ to kernels~$f$ which are
harmonic with respect to one variable. The system of functions in which the interpolating function is constructed will be defined from
restrictions of~$f$. This construction guarantees that the harmonicity of~$f$ is preserved for its interpolation error.
Hence, in order to achieve a prescribed accuracy in the exterior of a domain, it is sufficient to check it on its
boundary. This allows to construct $s_k$ in a kernel-independent and adaptive way.
The interpolating function~$s_k$ is then used to construct a quadrature rule which will be used in the construction of nested bases.
Sect.~\ref{sec:three} presents error estimates for functions $\e^{-\gamma|x|}$
based on radial basis functions. These results are used in Sect.~\ref{sec:appl} to derive exponential error estimates (via exponential sum approximation)
for $s_k$ when interpolating~$f(x,y)=|x-y|^{-\alpha}$ for arbitrary $\alpha>0$.
The goal of Sect.~\ref{sec:H2} is the construction of uniform $\H$- and $\H^2$-matrix approximations to matrices~\eqref{eq:matA} using the harmonic interpolants~$s_k$.
In Sect.~\ref{sec:num} we apply the new method to boundary integral formulations of Poisson boundary value
problems and to fractional diffusion problems and present numerical results which validate the presented method.


\section{Harmonic interpolants and quadrature rules} \label{sec:hiqr}
For the construction of $\H^2$-matrix approximations (see Sect.~\ref{sec:H2}), quadrature rules for the computation of integrals
\[
\int_X f(x,y)\ud x  
\]
will be required which depend only on the domain of integration~$X\subset\R^d$ and which are valid in the whole far-field of $X$, i.e.\ for $y\in \op{F}_\eta(X)$, where
\[
\op{F}_\eta(X):=\{y\in \R^d:\eta\,\dist(y,X)\geq \diam\,X\}
\]
with given $\eta>0$. Such quadrature formulas are usually based on polynomial interpolation
together with a-priori error estimates. The aim of this section is to introduce new adaptive
quadrature formulas which are controlled by a-posteriori error estimates.
In the special situation that $f(x,\cdot)$, $x\in X$, is harmonic in \[X^c:=\R^d\setminus \overline{X}\]
and vanishes at infinity it is possible to control the quadrature error for $y\in \op{F}_\eta(X)$ also computationally.
Notice that $f(x,y)=|x-y|^{-\alpha}$ is harmonic in $\R^d$, $d\geq3$, only for $\alpha=d-2$. Applying the following arguments in $\R^{d'+2}$, one can also treat the case $\alpha=d'$ for arbitrary $d'\in\N$.
Fractional exponents, which appear for instance in the case of the fractional Laplacian, will be treated in a forthcoming article.

Harmonic functions $u:\Omega\to\R$ in an unbounded domain $\Omega\subset\R^d$ are known to satisfy the mean value property
\[
u(x)=\frac{1}{|B_r|}\int_{B_r(x)} u(y)\ud y  
\]
for balls $B_r(x)\subset \Omega$ and the maximum principle
\[
\max_\Omega |u|\leq \max_{\partial \Omega} |u|  
\]
provided $u$ vanishes at infinity.

Let $\Sigma\subset\R^d$ be an unbounded domain such that (see Figure~\ref{fig:Sigma})
\begin{equation}\label{eq:ballxi}
\Sigma \supset  \op{F}_{\eta}(X)\quad \text{and}\quad \partial \Sigma\subset \op{F}_{2\eta}(X).
\end{equation}
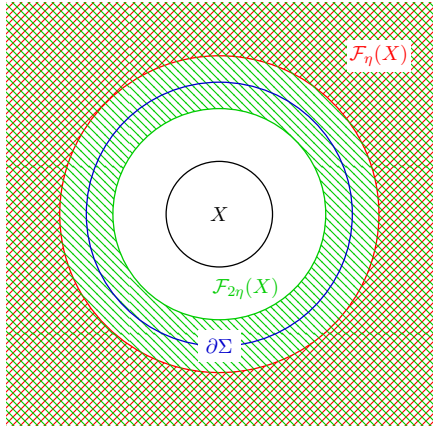
\begin{figure}[htb]
\begin{center}
\scalebox{0.7}{
\begin{tikzpicture} 
\fill [pattern = north east lines, pattern color=red] (-4,-4) rectangle (4,4);
\draw[line width=0.7pt, red, fill = white] (0,0) circle (3cm);
\fill [pattern = north west lines, pattern color=green!80!black] (-4,-4) rectangle (4,4);
\draw[line width=0.7pt, green!80!black, fill=white] (0,0) circle (2cm);
\draw[line width=0.7pt, blue!80!black] (0,0) circle (2.5cm);
\draw[line width=0.7pt, fill=white] (0,0) circle (1cm);
\node[rectangle] at (0,0) {$X$};
\fill [white] (2.4,2.7) rectangle (3.6,3.3);
\node[rectangle, red] at (3,3) {$\mathcal{F}_\eta(X)$};
\fill [white] (-0.4,-2.8) rectangle (0.4,-2.2);
\node[rectangle, blue!80!black] at (0,-2.5) {$\partial \Sigma$};
\node[rectangle, green!80!black] at (0.5,-1.4) {$\mathcal{F}_{2\eta}(X)$};
\end{tikzpicture}
}
\end{center}
\caption{$\Sigma$ and the far-fields $\op{F}_{2\eta}(X)$ and $\op{F}_\eta(X)$.} \label{fig:Sigma}
\end{figure}
A natural choice is $\Sigma=\op{F}_\eta(X)$. 
Since our aim is to check the actual accuracy and we cannot afford to inspect it on an infinite set,
we introduce the finite set $M\subset \partial \Sigma$ to be close to $\partial\Sigma$, i.e., we assume that $M$
satisfies
\begin{equation}\label{eq:Ydicht}
  \dist(y,M)\leq\delta,\quad y\in \partial \Sigma.
\end{equation}
In \cite{Bebendorf:2008} we have already used the following recursive definition for the construction of
an interpolating function~$s_k$ in the convergence analysis of the \textit{adaptive cross approximation}~\cite{MR2001j:65022}. 
Let $r_0=f$ and for $k=0,1,2,\dots$ assume that $r_k$ has already been defined. Let
$x_{k+1}\in X$ be chosen such that 
\begin{equation} \label{eq:unisol}
 r_k(x_{k+1},\cdot)\neq 0 \quad \textnormal{in } M,
\end{equation}
then set
\begin{equation}\label{eq:defrk}
r_{k+1}(x,y):=r_k(x,y)-\frac{r_k(x_{k+1},y)}{r_k(x_{k+1},y_{k+1})}\,r_k(x,y_{k+1})
\end{equation}
and $s_{k+1}:=f-r_{k+1}$, where $y_{k+1}\in M$ denotes the maximum of $|r_k(x_{k+1},\cdot)|$ in $M$.

It can be shown (see \cite{Bebendorf:2008}) that $s_k$ interpolates~$f$ at the chosen nodes~$x_i$, $i=1,\dots,k$, for all $y\in \op{F}_\eta(X)$, i.e.,
\[s_k(x_i,y)=f(x_i,y),\quad i=1,\dots,k,\]
and belongs to $F_k:=\spann\{f(\cdot,y_1),\dots,f(\cdot,y_k)\}$.
In addition, the choice of $(x_k,y_k)\in X\times M$ guarantees unisolvency, which can be seen from
\[
\deta C_k=r_0(x_1,y_1)\cdot\ldots\cdot r_{k-1}(x_k,y_k)\neq 0,
\]
where $C_k\in\R^{k\times k}$ denotes the matrix with the entries $(C_k)_{ij}=f(x_i,y_j)$, $i,j=1,\dots,k$.
Hence, one can define the Lagrange functions for the system 
and the nodes $x_i$, i.e.\ $L^{(j)}_k(x_i)=\delta_{ij}$, $i,j=1,\dots,k$, as
\[
L_k^{(i)}(x):=\frac{\deta C^{(i)}_k(x)}{\deta C_k}\in F_k,\quad i=1,\dots,k,
\]
where $C^{(i)}_k(x)\in\R^{k\times k}$ results from $C_k$ by replacing its $i$-th row with the
vector \[v_k(x):=\begin{bmatrix}f(x,y_1)\\ \vdots\\f(x,y_k)\end{bmatrix}.\]
Another representation of the vector $L_k\in\R^k$ of Lagrange functions $L_k^{(i)}$ is
\begin{equation}\label{eq:Lagf}
L_k(x)=C_k^{-T}v_k(x).
\end{equation}
Due to the uniqueness of the interpolation, $s_k$ has the representation 
\begin{equation}\label{eq:repsk}
s_k(x,y)=\sum_{i=1}^k f(x_i,y) L^{(i)}_k(x)=v_k(x)^TC_k^{-1}w_k(y),
\end{equation}
where $w_k(y):=[f(x_1,y),\dots,f(x_k,y)]^T$.

For an adaptive procedure it remains to control the interpolation error $f-s_k=r_k$ in $X\times \op{F}_\eta(X)$.
The following obvious property follows from \eqref{eq:defrk} via induction.
\begin{lemma}\label{lem:one}
  If $f(x,\cdot)$ is harmonic in $X^c$ and vanishes at infinity for all $x\in X$, then so do $s_k(x,\cdot)$ and~$r_k(x,\cdot)$.
\end{lemma}

The following lemma shows that although $M\subset\partial\Sigma$ is a finite set, it can be used to find an upper bound on the
maximum of $r_k(x,\cdot)$ in the unbounded domain $\op{F}_\eta(X)$.

\begin{lemma}\label{lem:two}
Let the assumptions of Lemma~\ref{lem:one} be valid and let $2q\eta\,\delta<\diam\,X$, where $q=(\sqrt[d]{2}-1)^{-1}+2$.
Then there is $c_k>0$ such that for $x\in X$ it holds
\[\max_{y\in \op{F}_\eta(X)} |f(x,y)-s_k(x,y)|
\leq 2\max_{y\in M} |f(x,y)-s_k(x,y)|+c_kq\delta,\]
where $c_k:=\norm{\nabla_y r_k(x,\cdot)}_\infty$.
\end{lemma}
\begin{proof}
  Let $x\in X$ and $y\in \partial \Sigma$. 
  We define the set \[N:=\{z\in B_{q\delta}(y):r_k(x,z)=0\}\] of zeros in $B_{q\delta}(y)$.
  If $N\neq\emptyset$ then with $z\in N$
  \[
    |r_k(x,y)|= |\int_0^1 (y-z)\cdot\nabla_y r_k(x,z+t(y-z))\ud t|\leq c_kq\delta.
  \]
In the other case $N=\emptyset$, our aim is to find $y'\in M$ such that $|r_k(x,y)|\leq 2|r_k(x,y')|$. $r_k$ does not change its sign and is harmonic
in $B_{q\delta}(y)$ due to $B_{q\delta}(y)\subset X^c$, which follows from \eqref{eq:ballxi} as
  \[
    2\eta\,\dist(B_{q\delta}(y),X)\geq 2\eta\,\dist(y,X)-2\eta q\delta\geq \diam\,X-2\eta q\delta>0.
    \]
Due to the assumption \eqref{eq:Ydicht} we can find $y'\in B_\delta(y)\cap M$.
Then $B_{(q-2)\delta}(y)\subset B_{(q-1)\delta}(y')\subset B_{q\delta}(y)$.
Hence, the mean value property (applied to $r_k$ if $r_k$ is positive or to $-r_k$ if $r_k$ is negative) shows
\begin{align*}
|r_k(x,y)|&=\frac{1}{|B_{(q-2)\delta}|}\int_{B_{(q-2)\delta}(y)} |r_k(x,z)|\ud z
\leq \frac{1}{|B_{(q-2)\delta}|}\int_{B_{(q-1)\delta}(y')} |r_k(x,z)|\ud z\\
&=\frac{|B_{(q-1)\delta}|}{|B_{(q-2)\delta}|}|r_k(x,y')|
=\left(\frac{q-1}{q-2}\right)^d|r_k(x,y')|
=2|r_k(x,y')|.
\end{align*}
Sine $r_k$ vanishes at infinity, \eqref{eq:ballxi} together with the maximum principle shows
\[
  \max_{y\in \op{F}_\eta(X)} |r_k(x,y)|\leq \max_{y\in \Sigma} |r_k(x,y)|\leq \max_{y\in \partial \Sigma} |r_k(x,y)|
\leq 2\max_{y'\in M} |r_k(x,y')|+c_kq\delta.
\]
\end{proof}

Notice that due to \eqref{eq:repsk} we have
\[
\nabla_y r_k(x,y) =\nabla_y f(x,y)-\nabla_y s_k(x,y)=\nabla_y f(x,y)-\sum_{i=1}^k L_k^{(i)}(x)\nabla_y f(x_i,y).
\]
Hence, \[c_k=\norm{\nabla_y r_k(x,\cdot)}_\infty\leq (1+\Lambda_k)\max_{x\in X}\norm{\nabla_y f(x,\cdot)}_\infty\]
with the Lebesgue constant $\Lambda_k(x):=\sum_{i=1}^k|L_k^{(i)}(x)|$.
Although it seems that $\Lambda_k(x)\sim k$ in practice, there is no proof for this observation up to now.
A related topic in interpolation theory are \textit{Leja points}; see \cite{MR0100726}.

To see that this special kind of interpolation is more efficient than polynomial interpolation,
we present the following example.
\begin{example}
Let $X\subset\R^3$ be 1000 points forming a uniform mesh of the unit cube centered at the origin. We choose $\Sigma=\{x\in\R^3:|x|>3\}$.
$M$ is a discretization of $\partial\Sigma$ with 768~points. We consider $f(x,y)=|x-y|^{-1}$
and compare the quality of $s_k$ with the quality of the interpolating tensor Chebyshev  polynomial of degree~$k$.
The following table shows the maximum pointwise error measured at $X$ and at three times as many points as $M$ has.
\begin{table}[htb]\centering
\begin{tabular}{c|ccccc}
 $k$  & 1 & 8 & 27 & 64 & 125 \\ \hline 
  Cross approximation & 3.28e-1 & 5.90e-2 & 5.8e-3 & 2.22e-4 & 1.12e-5\\
  Chebyshev interpolation & 4.55e-1 & 8.73e-2 & 2.18e-2 & 5.72e-3 & 2.10e-3 
\end{tabular}
\caption{Approximation error of $s_k$ and tensor Chebyshev polynomial of degree $k$.}
\end{table}
\end{example}

\subsection{Exponential error estimates for multivariate interpolation}\label{sec:three}
For analyzing the error of the cross approximation, the remainder $r_k$ has to be estimated. The proof in~\cite{Bebendorf:2008} establishes a connection
of $r_k$ with the best approximation in an arbitrary system~$\Xi = \{ \xi_1, \dots, \xi_k\}$ of functions. There, qualitative estimates are presented for a polynomial
system $\Xi$. For the uniqueness of polynomial interpolation it has to be assumed that the 
Vandermonde matrix $[\xi_j(x_i)]_{ij} \in \R^{k \times k}$ is non-singular.
The goal of the following section is to provide new error estimates for the convergence of cross
approximation which avoid the unisolvency assumption by employing radial basis function 
interpolation. Furthermore, we will be able to state a rule for choosing the next pivotal point $x_k$ (in addition to \eqref{eq:unisol})
leading to fast convergence rates.

Let $\kappa:\R^d\to\R$ be a continuous function. In the following we assume that $\kappa$ is positive definite,
i.e. 
\[
\int_{\R^d\times\R^d}\kappa(x-y)\fie(x)\overline{\fie(y)}\ud x\ud y>0  
\]
for all $0\neq \fie\in C_0^\infty(\R^d)$.
The Fourier transform of such functions determines a measure $\mu$ on $\R^d\setminus\{0\}$ such that
\[
\int\kappa(x)\fie(x)\ud x=\int\hat\fie(\xi)\ud\mu(\xi),\quad \fie\in C_0^\infty(\R^d).
\]
Following~\cite{mn92} we define $\mathscr{C}_\kappa$ the set of continuous functions~$f$ satisfying
\begin{equation}\label{eq:defCk}
(f,\fie)_{L^2}^2\leq c^2\int_{\R^d\times\R^d} \kappa(x-y)\,\fie(x)\,\overline{\fie(y)}\ud x\ud y
\end{equation}
for some constant $c>0$ and all $\fie\in C_0^\infty(\R^d)$. The smallest constant~$c$ in~\eqref{eq:defCk}
defines a norm~$\norm{f}_\kappa$ and $\mathscr{C}_k$ is a Hilbert space. 

Given a set $X_k := \{x_1,\dots,x_k\}\subset X$ consisting of $k \in \N$ nodes~$x_j$, an interpolant 
 $p \in \textnormal{span}\{\kappa(\cdot-x_j),\,j=1,\dots,k\}$ has to fulfill the conditions \[p(x_j) = f(x_j), \quad j = 1,\dots,k.\]
A solution of this interpolation problem can be written in its Lagrangian form
\begin{equation*}
 p(x) := \sum_{i = 1}^k f(x_i) \, L^{\kappa}_{i}(x),
\end{equation*}
where $L^{\kappa}_{i}(x) = \sum_{j = 1}^{k} \alpha_j^{(i)} \kappa(x-x_j)$ denote the Lagrange functions 
satisfying $L^{\kappa}_j(x_i) = \delta_{ij}$, i.e., its coefficients $\alpha^{(i)}\in\R^k$ are defined as the solution of the
linear systems of equations $A\alpha^{(i)} = e_i$ with $A := [\kappa(x_i-x_j)]_{ij}\in\R^{k\times k}$.
The error between a function $f\in\mathscr{C}_\kappa$
and its interpolant~$p$ is typically measured in terms of the \emph{fill distance}
\begin{equation*}
 h_{X_k,X} := \sup_{x \in X} \dist(x,X_k).
\end{equation*}
The following result is proved in \cite{mn92}.
\begin{theorem} \label{thm:RBFconv}
Let $X$ be a cube of side $b_0$.  Suppose that $\mu$ satisfies
  \begin{equation}\label{eq:mudec}
    \int |\xi|^k\ud\mu(\xi)\leq \rho^k k!,\quad k\in\N,
  \end{equation}
  for some $\rho>0$. Then there is $0<\lambda<1$ such that for all $f\in\mathscr{C}_\kappa$ the corresponding interpolant~$p$
 satisfies
 \[
   |f(x)-p(x)|\leq \lambda^{1/h_{X,X_k}}\norm{f}_\kappa
   \]
   for all $x\in X$.
\end{theorem}

\begin{remark}
  The assumption that $X$ is a cube can be generalized. Theorem~\ref{thm:RBFconv} remains valid as long as $X$ can be expressed as the union of rotations and translations of a fixed cube of side $b_0$. Actually, any ball in 
  $\R^d$ or any set X with sufficiently smooth boundary fulfills the requirements.
 \end{remark}

Elements $f\in \mathscr{C}_\kappa$ can be characterized (see~\cite{mn88, mn90}) by the existence of a function $g\in L^2_\mu$ such that
\begin{equation}\label{eq:charCk}
\hat f(\xi)\ud\xi=g(\xi)\ud \mu(\xi).  
\end{equation}
For later purposes we prove
\begin{lemma} \label{lem:expappr}
  Let $\kappa(x)=\exp(-\beta|x|^2)$  with $\beta>0$.
  Then $\kappa$ is positive definite and the measure~$\mu$ associated with $\kappa$ satisfies~\eqref{eq:mudec}.
  Furthermore, $h(x)=\exp(-\gamma |x|)$ with $\gamma>0$ belongs to
  $\mathscr{C}_\kappa$.
\end{lemma}
\begin{proof}
  Since the Fourier transform of a Gauss function is again a Gauss function,
  the measure associated with $\kappa$ is
  \[
  \ud \mu(\xi)= \left( \frac{\pi}{\beta}\right)^{d/2} \exp\left( -\frac{|\xi|^2}{4\beta} \right)\ud\xi.  
  \]
  $\mu$ satisfies \eqref{eq:mudec}.
    Let $H(r)=\exp(-\gamma r)$ with $r=|x|$. Then $\hat h(\xi)=\hat H(s)$, where $s=|\xi|$.
    Since \[
    \hat H(s)=(2\pi)^{d/2}s^{(2-d)/2}\int_0^\infty J_{d/2-1}(sr)\,r^{d/2}\,H(r)\ud r 
    \]
    with the Bessel function $J_{d/2-1}$ of order $d/2-1$, we obtain  for the Hankel transform (cf.~\cite{be54}) that
  \begin{align*}
    \hat H(s) &= (2\pi)^{d/2}s^{(1-d)/2}\int_0^\infty r^{d/2 - 1 +1/2} \exp(-\gamma r)\, J_{d/2-1}(sr) \, (sr)^{1/2}\ud r \\
    &=  (2\pi)^{d/2}s^{(1-d)/2} \pi^{-1/2} 2^{d/2} \, \Gamma\left(\frac{d+1}{2}\right) s^{(d-1)/2} \frac{\gamma}{(\gamma^2 + s^2)^{(d+1)/2}} \\
    &= 2^d \pi^{(d-1)/2} \, \Gamma\left(\frac{d+1}{2}\right) \frac{\gamma}{(\gamma^2 + s^2)^{(d+1)/2}}
   \end{align*}
    and
    \[
    \hat h(\xi)=2^d \pi^{(d-1)/2} \, \Gamma\left(\frac{d+1}{2}\right) \frac{\gamma}{(\gamma^2 + |\xi|^2)^{(d+1)/2}}  ,
    \]
    where $\Gamma$ denotes the Gamma function.
    Defining the $L^2_\mu$-function 
    \[
    g(\xi)= 2^d \beta^{d/2} \pi^{-1/2} \, \Gamma\left(\frac{d+1}{2}\right)  \frac{\gamma}{(\gamma^2 + |\xi|^2)^{(d+1)/2}} \exp\left( \frac{|\xi|^2}{4\beta} \right)
    \] we obtain \eqref{eq:charCk}, because
  \[
\int_0^\infty |\hat H(s)|^2 s^{d-1} \ud s = 2^{2d} \pi^{d-1}  \, \Gamma^2\left(\frac{d+1}{2}\right) \gamma^2  \int_0^\infty  \frac{s^{d-1}}{(\gamma^2 + s^2)^{d+1}} \ud s < \infty.
    \]
    \end{proof}

\subsection{Application to $|x-y|^{-\alpha}$} \label{sec:appl}
We consider functions $f$ of the form
\begin{equation*}
 f(x,y) = \frac{1}{|x-y|^\alpha}, \quad \alpha > 0,
\end{equation*}
on two domains $X,Y$ satisfying
\begin{equation}\label{eq:adm}
\max\{\diam\,X,\diam\,Y\}\leq \eta\,\dist(X,Y).
\end{equation}
The validity of the latter condition will result from a partitioning of the
computational domain $\Omega\times\Omega$ induced by a hierarchical partitioning of the matrix~\eqref{eq:matA}.

Let $\kappa(x,y) = \exp(-\beta |x-y|^2)$. For fixed $y\in Y$ we interpolate $f$ with the radial
basis function
\begin{equation}\label{eq:kappip}
  p_y(x) := \sum_{i = 1}^k f(x_i,y)\, L_i^{\kappa}(x)
\end{equation}
on the data set $X_k = \{x_1, \dots, x_k\}$. Here, $L_j^{\kappa}$, $j = 1,\ldots,k$, are the Lagrange functions for $\kappa$ and $X_k$.
\begin{lemma} \label{lem:five} 
  Let $\sigma:=\dist(X,Y)$. Then for $x\in X$, $y\in Y$
  \[
  |f(x,y)-p_y(x)|\leq (c+\Lambda_k^\kappa)\left(\frac{2}{\sigma}\right)^\alpha\lambda^{1/h_{X_k,X}},
  \]
  where $\Lambda^\kappa_k:=\sup_{x\in X}\sum_{i=1}^k|L_i^\kappa(x)|$ denotes the Lebesgue constant.
\end{lemma}
\begin{proof}
  Functions of type $f$ are not covered by Theorem~\ref{thm:RBFconv}. Therefore, we
  additionally employ \textit{exponential sum approximations}
  \[
  g_r(t):=\sum_{j=1}^r \omega_j \exp(-\gamma_j \, t)
  \]
  of $g(t):=t^{-\alpha}$ with finite~$r$ on the interval $[1,R]$ in order to approximate~$f$. According to \cite{BrHa09}, there are coefficients~$\omega_j,\gamma_j>0$
  such that
  \[
  \norm{g-g_r}_{L^\infty[1,R]}\leq 8\cdot 2^\alpha \exp\left(-\frac{\pi^2 r}{\log(8R)}\right).
  \]
   Choosing $r$ such that
  \[
    8\exp\left(-\frac{\pi^2 r}{\log(8+16\eta)}\right)=\lambda^{1/h_{X_k,X}}
  \]
  and $R=1+2\eta$, \eqref{eq:adm} implies for $x\in X$ and $y\in Y$
  \[
  1\leq t:= \frac{|x-y|}{\sigma}\leq \frac{\diam\,X+\sigma+\diam\,Y}{\sigma}\leq 1+2\eta=R.
  \]
  Letting $h_{j,y}(x)=\sigma^{-\alpha}\exp(-\gamma_j|x-y|/\sigma)$, we obtain
\[
|f(x,y)-\sum_{j=1}^r\omega_j h_{j,y}(x)|
=\sigma^{-\alpha}|g(t)-g_r(t)|\leq  8\left(\frac{2}{\sigma}\right)^\alpha \exp\left(-\frac{\pi^2 r}{\log(8+16\eta)}\right).
\]
According to Theorem~\ref{thm:RBFconv} and Lemma~\ref{lem:expappr}, the functions $h_{j,y}$ can be
interpolated using the radial basis function $\kappa$ 
on the data set $X_k = \{x_1, \dots, x_k\}$, i.e.
\[
\norm{h_{j,y}-\tilde h_{j,y}}_{\infty, X} \leq \lambda^{1/h_{X_k,X}}\norm{h_{j,y}}_\kappa,
\]
where \[
 \tilde h_{j,y}(x) = \sum_{i = 1}^k h_{j,y}(x_i)\, L_i^{\kappa}(x).
\]
Let $h^*(x):=\sigma^{-\alpha}\sup_{y\in Y}\exp(-\beta_*|x-y|)$, where $\beta_*:=\min_{j=1,\dots,r}\gamma_j/\sigma$. From
\begin{equation*}
(h_{j,y}, \fie)_{L^2}^2 \leq (h^*,\fie)_{L^2}^2\leq \norm{h^*}_\kappa^2\int_{\R^d\times\R^d} \kappa(x-z)\,\fie(x)\,\overline{\fie(z)}\ud x\ud z, \quad 1 \leq j \leq r,
\end{equation*}
for all $\fie\in C_0^\infty(\R^d)$ we obtain that $\norm{h_{j,y}}_\kappa\leq\norm{h^*}_\kappa$.
Hence,
\[
  \norm{\sum_{j=1}^r \omega_jh_{j,y}-\sum_{j=1}^r\omega_j\tilde h_{j,y}}_{\infty, X} \leq \lambda^{1/h_{X_k,X}}\norm{h^*}_\kappa \sum_{j=1}^r\omega_j.
\]
Notice that $\sum_{j=1}^r \omega_j\leq \e^{\gamma_*}\sum_{j=1}^r \omega_j\e^{-\gamma_j}=\e^{\gamma_*}g_r(1)\leq c$, where
$\gamma_*=\max_{j=1,\dots,r}\gamma_j$.
The last step is to show that
\begin{align*}
\norm{p_y-\sum_{j=1}^r \omega_j \tilde h_{j,y}}_\infty&=\norm{\sum_{i=1}^k [f(x_i,y)-\sum_{j=1}^r \omega_j h_{j,y}(x_i)]L_i^\kappa}_\infty\\
&\leq \sup_{x\in X} \sum_{i=1}^k |f(x_i,y)-\sum_{j=1}^r \omega_j h_{j,y}(x_i)|\,|L_i^\kappa(x)|\\
& \leq 8\left(\frac{2}{\sigma}\right)^\alpha \exp\left(-\frac{\pi^2 r}{\log(8+16\eta)}\right)\Lambda^\kappa_k\\
&=\left(\frac{2}{\sigma}\right)^\alpha \lambda^{1/h_{X_k,X}}\Lambda^\kappa_k.
\end{align*}
The assertion follows from the triangle inequality.
\end{proof}

The convergence can be controlled by choosing the node $x_{k+1}$ such that
the fill distance~$h_{X_{k+1},X}$ is minimized from step~$k$ to step~$k+1$. This minimization problem can be solved efficiently, i.e.\ with logarithmic-linear complexity, with
the approximate nearest neighbor search described in~\cite{am93, am95, amnsw98}.
Since we can expect that the fill distance behaves like $h_{X_k,X} \sim k^{-1/d}$, Lemma~\ref{lem:five} shows exponential convergence of $p_y$
with respect to~$k$ provided the Lebesgue constant grows sub-exponentially.

Applying the results of the previous lemma to the remainder~$r_k$, we obtain the following result for interpolating~$f$
on $X\times Y$.
\begin{theorem}\label{thm:approxErr}
 For $y \in Y$ let $p_y$ denote the radial basis function interpolant \eqref{eq:kappip} for
 $f_y:= f(\cdot,y)=|\cdot-y|^{-\alpha}$.
Choosing $y_1,\dots,y_k\in Y$ such that
\[
  |\deta C_k^{(i)}(y)|\leq   c_M|\deta C_k|, \quad 1 \leq i \leq k, \, y \in Y,
\]
where $c_M>1$ is a constant, it holds that
\[
 |r_k(x,y)| \leq c(c_Mk+1) \,\lambda^{1/h_{X_k,X}},
\]
where $X_k := \{x_1,\ldots,x_k\}$.
\end{theorem}
\begin{proof}
Let the vector of the Lagrange functions $L_{i}^{\kappa}$, $i = 1,\ldots,k$, corresponding to the radial basis function~$\kappa$ and the nodes $x_1,\ldots,x_k$ be 
given by
\begin{equation*}
 L^{\kappa}(x) = \begin{bmatrix} L_1^{\kappa}(x) \\ \vdots \\ L_k^{\kappa}(x) \end{bmatrix}.
\end{equation*}
Using \eqref{eq:repsk}, we obtain
\begin{align*}
 r_k(x,y) &= f(x,y) - v_k(x)^T C_k^{-1} w_k(y)\\
 &= f(x,y) - w_k(y)^T L^\kappa(x) - \left[ v_k(x) - C_k L^\kappa(x)\right]^T C_k^{-1} w_k(y) \\
 &= f_y(x) - p_y(x) - \sum_{i = 1}^k [C_k^{-1} w_k(y)]_i \, [f_{y_i}(x) - p_{y_i}(x)] \\
 &= f_y(x) - p_y(x) - \sum_{i = 1}^k \frac{\deta C_k^{(i)}(y)}{\deta C_k} [f_{y_i}(x) - p_{y_i}(x)], 
\end{align*}
where the last line follows from Cramer's rule. The assertion follows from the triangle inequality and Lemma~\ref{lem:five}.
\end{proof}

\begin{remark}
  Choosing the nodes $y_1,\ldots,y_k$ according to the condition
  \[
  |r_{k-1}(x_k,y_k)| \geq |r_{k-1}(x_k,y)| \quad \textnormal{for all } y \in Y,
  \]
  which is much easier to check in practice, leads to the estimate
  \[
    |\deta C_k^{(i)}(y)| \leq 2^{k-i}|\deta C_k|,\quad 1\leq i\leq k,\,y\in Y;
  \]
  for details see \cite{Bebendorf:2008}.
\end{remark}

\section{Construction of $\mathcal{H}^2$-matrix approximations} \label{sec:H2}
The aim of this section is to construct hierarchical matrix approximations to the matrix~$A$ defined
in~\eqref{eq:matA}. To this end, we first partition the set of indices
$I\times J$, $I=\{1,\dots,M\}$ and $J=\{1,\dots,N\}$, into sub-blocks $t\times s$, $t\subset I$
and $s\subset J$, such that the associated supports
\[X_t:=\bigcup_{i\in t} \supp\,\fie_i\quad \text{and}\quad
Y_s:=\bigcup_{j\in s}  \supp\,\psi_j
\]
satisfy
\begin{equation}\label{eq:admcl1}
  \eta\,\dist(X_t,Y_s)\geq \max\{\diam\,X_t,\diam\,Y_s\},
\end{equation}
i.e.\ $Y_s\subset \op{F}_\eta(X_t)$ and $X_t\subset \op{F}_\eta(Y_s)$. Notice that from Sect.~\ref{sec:appl}
we know that the singular part~$f$ of the kernel function~$K$ in~\eqref{eq:matA} can be approximated on the pair $X_t\times Y_s$.

The usual way of constructing such partitions is based on
\textit{cluster trees}; see \cite{MR2001i:65053,Bebendorf:2008}.
A cluster tree~$T_I$ for the index set $I$ is a binary tree with root $I$, where
each $t \in T_I$ and its nonempty successors $S_I(t)=\{t',t''\}\subset T_I$ (if they exist)
satisfy $t = t' \cup t''$ and  $t' \cap t'' = \emptyset$.
We refer to $\mathcal{L}(T_I) = \{t \in T_I:S_I(t)=\emptyset\}$ as the leaves of $T_I$ and define 
\[T_I^{(\ell)} = \{t \in
T_I : \dist(t,I) = \ell \} \subset T_I,
\] where $\dist(t, s)$
is the minimum distance between $t$ and $s$ in $T_I$. 
Furthermore, \[L(T_I):=\max\{\dist(t,I),\,t\in T_I\}+1\] denotes the depth of $T_I$.

Once the cluster trees $T_I$, $T_J$ for the index sets $I$ and $J$ have been computed, a
partition~$P$ of~$I\times J$ can be constructed from it. 
A block cluster tree~$T_{I\times  J}$ is a quad-tree with root
$I\times J$ satisfying conditions analogous to a cluster tree.
It can be constructed from the
cluster trees $T_I$ and $T_J$ in the following way. Starting from the
root $I\times J\in T_{I\times J}$, let the sons of a block $t\times
s\in T_{I\times J}$ be $S_{I\times J}(t,s):=\emptyset$ if
$t\times s$ satisfies \eqref{eq:admcl1} or $\min\{|t|,|s|\}\leq \nmin$ with a given constant~$\nmin>0$.
In the remaining case, we set $S_{I\times J}(t,s):=S_I(t)\times S_J(s)$.
The set of leaves of $T_{I\times J}$ defines a partition~$P$ of~$I\times J$ and its cardinality $|P|$ is
of the order $\min\{|I|,|J|\}$; see \cite{Bebendorf:2008}. As usual, we partition $P$ into admissible and
non-admissible blocks \[
P=P_{\adm}\cup P_{\nadm},
\]
where each $t\times s\in P_{\adm}$ satisfies \eqref{eq:admcl1} and
each $t\times s\in P_{\nadm}$ is small, i.e.\ satisfies
$\min\{|t|,|s|\}\leq \nmin$.

\subsection{Uniform $\H$-matrix approximation}
Hierarchical matrices are well-suited for treating non-local operators with logarithmic-linear complexity; see~\cite{Bebendorf:2008,Bo10,WH15}.
\begin{definition}
  A matrix $A\in\R^{I\times J}$ satisfying $\rank\,A|_b\leq k$ for all $b\in P_\adm$
  is called \emph{hierarchical matrix ($\H$-matrix)} of blockwise rank at most~$k$.
\end{definition}

In order to approximate the matrix~\eqref{eq:matA} more efficiently, we employ uniform $\H$-matrices; see~\cite{MR2000c:65039}.
\begin{definition}  
  A \emph{cluster basis} $\Phi$ for the rank distribution $(k_t)_{t\in T_I}$ is a family 
   $\Phi=(\Phi(t))_{t\in T_I}$ of matrices $\Phi(t) \in \R^{t\times k_t}$.
 \end{definition}

\begin{definition}
 Let $\Phi$ and $\Psi$
  be cluster bases for $T_I$ and $T_J$. A matrix $A\in\R^{I\times J}$ satisfying
  \[
  A|_{ts}=\Phi(t) \, F(t,s) \, \Psi(s)^H\quad
  \text{for all }t\times s\in P_\adm
  \]
  with some $F(t,s)\in\R^{k_t^\Phi\times k_s^\Psi}$ is called
  \emph{uniform hierarchical matrix} for $\Phi$ and $\Psi$.
\end{definition}

The storage required for the coupling matrices~$F(t,s)$ is of the order~$k\min\{|I|,|J|\}$ if for the sake of simplicity it is assumed that $k_t\leq k$ for all $t\in T_I$. Additionally, it is not useful to choose~$k_t>|t|$.
The cluster bases~$\Phi$ and~$\Psi$ require $k[|I|L(T_I)+|J|L(T_J)]$ units of storage; see~\cite{HAKHSA00}.

In the following we employ the method from Sect.~\ref{sec:hiqr} to construct a uniform $\H$-matrix approximation to
an arbitrary block $t\times s\in P_\adm$ of matrix~\eqref{eq:matA}. Let $\varepsilon > 0$ be given and
$[x]_t=\{x^t_p,\,p\in\tau_t\}\subset X_t$ and $[v]_t=\{v^t_p,\,p\in\sigma_t\}\subset\op{F}_\eta(X_t)$ be the pivots chosen in~\eqref{eq:defrk}
 such that
 \begin{equation}\label{eq:interpolf}
  |f(x,y)-\sum_{p\in \tau_t} L^t_p(x)f(x^t_p,y)|< \eps,\quad x\in X_t,\,y\in\op{F}_\eta(X_t),
 \end{equation}
 for each cluster~$t$. Here, $L^t(x):=f(x,[v]_t)f^{-1}([x]_t,[v]_t)$ denotes the vector of Lagrange functions defined in~\eqref{eq:Lagf}.  $\tau_t$ and $\sigma_t$ denote index sets with cardinality~$k$.
 From Theorem~\ref{thm:approxErr} we know that $k\sim|\log\eps|^d$.
 Similarly, for $s\in T_J$ let $[y]_s=\{y^s_q,\,q\in\sigma_s\}\subset Y_s$ and $[w]_s=\{w^s_q,\,q\in\tau_s\}\subset\op{F}_\eta(Y_s)$ be chosen such that
 \begin{equation}\label{eq:interpolg}
  |f(x,y)-\sum_{q\in \sigma_s}f(x,y^s_q)L^s_q(y)|< \eps,\quad x\in \op{F}_\eta(Y_s),\,y\in Y_s,
 \end{equation}
 where $L^s(y):=f^{-1}([w]_s,[y]_s)f([w]_s,y)$.
For $x\in X_t$ and $y\in Y_s$ this yields the dual interpolation
 \[f(x,y)\approx\sum_{p\in \tau_t} L^t_p(x)f(x^t_p,y)\approx \sum_{p\in\tau_t,\,q\in\sigma_s} L^t_p(x)\, f(x^t_p,y^s_q) \, L^s_q(y)\]
with corresponding interpolation error
  \begin{align}
|f(x,y)-\sum_{p\in\tau_t,\,q\in\sigma_s} L^t_p(x)\,f(x^t_p,y^s_q) \,L^s_q(y)|
&\leq |f(x,y)-\sum_{p\in\tau_t} L^t_p(x)f(x^t_p,y)|+\nonumber\\
&\quad +\sum_{p\in\tau_t} |L^t_p(x)||f(x^t_p,y)-\sum_{q\in\sigma_s} f(x^t_p,y^s_q) L^s_q(y)|\nonumber\\
&\leq \eps+\eps\sum_{p\in\tau_t} |L_p^t(x)|=(1+\Lambda_k^t)\eps \label{eq:errdo}
\end{align}
and the Lebesgue constant $\Lambda_k^t\geq 1$.
We define the rank-$k^2$ matrix
\begin{equation}\label{eq:matB}
  b_{ij}=\sum_{p\in\tau_t,\,q\in\sigma_s} f(x_{p}^t,y_{q}^s) \int_{X_t}  L^t_{p}(x)\fie_i(x)\xi(x) \ud x\int_{Y_s} L^s_{q}(y)\psi_j(y) \zeta(y)\ud y
  =[\Phi(t) \, F(t,s)\, \Psi(s)^T]_{ij},
\end{equation}
where $\xi$ and $\zeta$ are the functions defined in \eqref{eq:matA1}. Notice that
both matrices \[[\Phi(t)]_{ip}:=\int_{X_t}  L^t_{p}(x)\fie_i(x) \xi(x)\ud x\quad\text{and}
\quad [\Psi(s)]_{jq}:=\int_{Y_s} L^s_q(y)\psi_j(y) \zeta(y)\ud y\]
are associated only with $t$ and $s$, respectively, and can be precomputed independently of each other.
Only the matrix $F(t,s)\in\R^{k\times k}$ with $[F(t,s)]_{pq}:=f(x^t_p,y^s_q)$ depends on both clusters $t$ and $s$.
\begin{remark}
Since the vector of Lagrange functions $L^t(x)$ has the representation $L^t(x)=C_k^{-1}v_k(x)$,
the matrices $\Phi(t)\in\R^{t\times \tau_t}$ can be found from solving the linear system
\[
  C_k \Phi(t)=[\int_{X_t} v_k(x)\fie_i(x)\xi(x)\ud x]_i.
\]
\end{remark}
With $\norm{\fie_i}_{L^1}=1=\norm{\psi_j}_{L^1}$ the Cauchy-Schwarz inequality implies
\begin{align*}
|a_{ij}-b_{ij}|&\leq \int_{Y_s}\int_{X_t} |f(x,y)-\sum_{p\in \tau_t,\,q\in\sigma_s} L_p^t(x) \, f(x_p^t,y_q^s)\,L_q^s(y)|\, |\xi(x)|\,|\fie_i(x)|\,|\zeta(y)|\,|\fie_j(y)|\ud x\ud y\\
&\stackrel{\eqref{eq:errdo}}{\leq} 2\Lambda_k^t\,\norm{\xi}_\infty\norm{\zeta}_\infty\,\eps.
\end{align*}
and thus 
\begin{equation}\label{eq:apprxAPP}
  \begin{aligned}
\norm{A|_{ts}-B}_2^2 &\leq \norm{A|_{ts}-B}_F^2=\norm{A|_{ts}-\Phi(t)\, F(t,s)\, \Psi(s)^T}_F^2=\sum_{i\in t,\,j\in s} |a_{ij}-b_{ij}|^2\\
&\leq (2\Lambda_k^t \norm{\xi}_\infty\norm{\zeta}_\infty)^2 |t||s|\eps^2.  
  \end{aligned}
\end{equation}
Notice that the computation of the double integral for a single entry of the Galerkin matrix~\eqref{eq:matA}
is replaced with two single integrals in~\eqref{eq:matB}.

\subsection{Nested bases}
In order to reduce the amount of storage
for storing the bases $\Phi$ and $\Psi$ one can establish a recursive relation
among the basis vectors. The corresponding structure are $\H^2$-matrices; see \cite{HAKHSA00,Bo10}.
This sub-structure of $\mathcal{H}$-matrices is even mandatory if a logarithmic-linear complexity is to be achieved for high-frequency Helmholtz problems.
To this end, \emph{directional $\mathcal{H}^2$-matrices} have been introduced in~\cite{BKV12}.

\begin{definition}  
 A cluster basis $U=(U(t))_{t\in T_I}$ is called \emph{nested} if for each $t\in T_I\setminus\L{T_I}$ 
  there are transfer matrices $T_{t't}\in \R^{k_{t'}\times k_t}$ such that 
  for the restriction of the matrix $U(t)$ to the rows $t'$ it holds that
  \[
    U(t)|_{t'}=U(t')\,T_{t't}\quad \text{for all }t'\in S_I(t).
  \]
\end{definition}
For estimating the complexity of storing a nested cluster basis~$U$ notice
 that the set of leaf~clusters~$\L{T_I}$
constitutes a partition of~$I$ and for each
leaf cluster~$t\in\L{T_I}$ at most $k|t|$ entries have to be stored. Hence,
$\sum_{t\in \L{T_I}} k|t|=k|I|$ units of storage are required for the
leaf matrices $U(t)$, $t\in\L{T_I}$. 
The storage required for the transfer matrices is of the order $k|I|$, too; see \cite{HAKHSA00}.
\begin{definition}\label{def:H2}
A matrix $A\in\R^{I\times J}$ is called \emph{$\H^2$-matrix} if there are nested cluster bases~$U$
and~$V$ such that for $t\times s \in P_\adm$
  \[
    A|_{ts} = U(t) \, F(t,s) \, V^H(s)
  \]
  with coupling matrices $F(t,s)\in \R^{k_t^U\times k_s^V}$.
\end{definition}
Hence, the total storage required for an $\H^2$-matrix is of the order~$k (|I|+|J|)$. 
\begin{remark}
It may be advantageous to consider only nested bases for clusters~$t$
having a minimal cardinality $n_{\min}^{\H^2}\geq \nmin$. Blocks consisting of smaller clusters are treated with $\H$-matrices.
\end{remark}

We define the matrices $U(t)\in\R^{t\times k_t}$, $t\in T_I$, by the following recursion.
If $t\in T\setminus \op{L}(T_I)$ then the set of sons~$S_I(t)$ is non-empty and we define
\[
U(t)|_{t'}=U(t')\,T^U_{t't} ,\quad t'\in S_I(t),
\]
with the transfer matrix \[T^U_{t't}:=f([x]_{t'},[v]_t)f^{-1}([x]_{t},[v]_t)\in\R^{k_{t'}\times k_t}.\]
For leaf clusters $t\in \op{L}(T_I)$ we set $U(t)=\Phi(t)$.
Similarly, we define matrices $V(s)\in\R^{s\times k_s}$, $s\in T_J$, using transfer matrices
\[
  T^V_{s's}:=f^T([w]_s,[y]_{s'})f^{-T}([w]_{s},[y]_s)\in\R^{k_{s'}\times k_s}.
\]
Then $U:=(U(t))_{t\in T_I}$ and $V:=(V(t))_{t\in T_J}$ are nested bases.

\begin{lemma}
  Assuming that $\max_{t\in T_I} \{\norm{U(t)}_F,\norm{V(t)}_F,\norm{T_{t't}^U}_F\}\leq \gamma$ and $k_t\leq k$ it holds that there exists a constant $c > 0$ such that
  \[
  \norm{A|_{ts}-U(t)\, F(t,s)\, V(s)^T}_F\leq c (L-\ell)\sqrt{|t||s|}\,\norm{\xi}_\infty\norm{\zeta}_\infty\,\eps,\quad t\times s \in P_\adm,
  \]
  where $\ell$ denotes the level of $t\times s$.
\end{lemma}
\begin{proof}
  Let $t\in T_I\setminus \op{L}(T_I)$ and $s\in T_J\setminus \op{L}(T_J)$.
  For $t'\in S_I(t)$ and $s'\in S_J(s)$ we have
  \begin{equation}\label{eq:recU}
   \begin{aligned}
    U(t)|_{t'} F(t,s) V(s)|_{s'}^T&=U(t')T^U_{t't}F(t,s) (T^V_{s's})^T V(s')^T\\
    &=U(t')F(t',s') V(s')^T-U(t')D(t',s') V(s')^T,
  \end{aligned}
\end{equation}
  where $D(t',s'):=F(t',s')-T^U_{t't}F(t,s) (T^V_{s's})^T$.
  Using
  \[
    \norm{D(t',s')}^2_F\leq2\norm{F(t',s')-T_{t't}^UF(t,s')}_F^2+2\norm{T_{t't}^U}_F^2\norm{F(t,s')-F(t,s) (T^V_{s's})^T}_F^2,
  \]
  one observes that the previous expression consists of matrices with entries
  \[
  f(x_i,y_j)-f(x_i,[v]_t)f^{-1}([x]_{t},[v]_t)f([x]_{t},y_j),\quad i\in t',\,j\in s',
    \]
    and
    \[
  f(x_i,y_j)-f(x_i,[y]_s)f^{-1}([w]_{s},[y]_s)f([w]_{s},y_j),\quad i\in t,\,j\in s',
    \]
    which can be estimated using \eqref{eq:interpolf} and \eqref{eq:interpolg} due to $x_i\in X_t\subset\op{F}_\eta(Y_s)$ and $y_j\in Y_s\subset \op{F}_{\eta}(X_t)$.
Thus, \[\norm{D(t',s')}_F\leq \sqrt{2(1+\gamma^2)}\sqrt{|t'||s'|}\,\eps.\]
By induction we prove that $\norm{A|_{ts}-U(t) F(t,s) V(s)^T}_F\leq \gamma^2\sqrt{2(1+\gamma^2)}(L-\ell)\sqrt{|t||s|}\,\norm{\xi}_\infty\norm{\zeta}_\infty\,\eps$, where $\ell$ denotes the maximum of the levels of~$t$ and~$s$.
If both $t$ and $s$ are leaves, then $\norm{A|_{ts}-\Phi(t) F(t,s)\Psi(s)^T}\leq 2\Lambda_k^t \sqrt{|t||s|}\,\norm{\xi}_\infty\norm{\zeta}_\infty\,\eps$ due to~\eqref{eq:apprxAPP}.
From \eqref{eq:recU} we see
  \begin{align*}
    \norm{A|_{t's'}&-U(t)|_{t'} F(t,s) V(s)|_{s'}^T}_F
    \leq \norm{A|_{t's'}-U(t')F(t',s') V(s')^T}_F+\norm{U(t')D(t',s')V(s')^T}_F\\
    &\leq\gamma^2\sqrt{2(1+\gamma^2)}(L-\ell-1)\sqrt{|t'|\,|s'|}\,\norm{\xi}_\infty\norm{\zeta}_\infty\,\eps+\gamma^2\sqrt{2(1+\gamma^2)}\sqrt{|t'||s'|}\,\eps\\
    &\leq \gamma^2\sqrt{2(1+\gamma^2)}(L-\ell)\sqrt{|t'|\,|s'|}\,\norm{\xi}_\infty\norm{\zeta}_\infty\,\eps.
  \end{align*}
This shows
\begin{align*}
  \norm{A|_{ts}-U(t) F(t,s) V(s)^T}_F^2&= \sum_{t'\in S_I(t),\, s'\in S_J(s)}\norm{A|_{t's'}-U(t)|_{t'} F(t,s) V(s)|_{s'}^T}_F^2\\
  & \leq 2\gamma^4 (1+\gamma^2)(L-\ell)^2(\norm{\xi}_\infty\norm{\zeta}_\infty\,\eps)^2 \sum_{t'\in S_I(t),\, s'\in S_J(s)}|t'|\,|s'|\\
  &=2\gamma^4 (1+\gamma^2)(L-\ell)^2(\norm{\xi}_\infty\norm{\zeta}_\infty\,\eps)^2 |t||s|.
\end{align*}
The same kind of estimate holds if $t$ or $s$ is a leaf, because then $U(t)=\Phi(t)$ or $V(s)=\Psi(s)$.
\end{proof}

\section{Numerical results} \label{sec:num}
The focus of the following numerical tests lies on two problems. The first problem is an exterior boundary value problem for the Laplace equation, the second is a fractional diffusion process.
All tests compare the method presented in this article with an $\H$-matrix approximation
generated by \emph{adaptive cross approximation}~(ACA); see \cite{Bebendorf:2008}.
All computations were performed on a computer consisting of two Intel E5-2630~v4 processors.
The construction of the matrix approximation was done in parallel using 40~cores.

\subsection{Exterior boundary value problem} \label{sec:numLap}
We consider the Dirichlet boundary value problem for the Laplace equation in the exterior of the Lipschitz domain $\Omega \subset \R^3$, i.e.
\begin{equation} \label{mp}
 \begin{aligned}
  -\Delta u &= 0 \quad \text{in } \Omega^c:= \R^3 \setminus \overline{\Omega}, \\
  \gamma_0^{\textnormal{ext}}u &= g \quad \text{on } \partial\Omega,
 \end{aligned}
\end{equation}
where $\gamma_0^{\textnormal{ext}}$ denotes the exterior trace and $g$ the given Dirichlet data in the trace space $H^{1/2}(\partial\Omega)$ of the Sobolev space $H^1(\Omega^c)$. 
In order to guarantee that the problem is well-defined, we additionally assume suitable conditions at infinity.

Using the single and double layer potential operators
\[
  \op{V}\psi(x) := \int_{\partial \Omega} \psi(y)\, K(x-y) \ud s_y, \quad 
  \op{K}\phi(x) := \int_{\partial \Omega} \phi(y) \,\gamma^{\textnormal{ext}}_{1,y}K(x-y) \ud s_y,
\]
where 
\begin{equation*}
 K(x) = \frac{1}{4\pi} |x|^{-1},\quad x \in \R^3\ohne\{0\},
\end{equation*}
denotes the fundamental solution,
the solution of \eqref{mp} is given by the \textit{representation formula}
\[
u(x)=\op{V}\psi(x)-\op{K}g(x),\quad x\in\R^3\setminus\partial\Omega.
\]
The task is to compute the missing Neumann data $\psi := \gamma^{\textnormal{ext}}_{1}u \in H^{-1/2}(\partial \Omega)$ from the boundary integral equation 
\begin{equation} \label{bp}
 \op{V}\psi = \left(\frac{1}{2}\op{I}+\op{K}\right)g \quad \text{on } \partial \Omega.
\end{equation}
The unique solvability of the boundary integral equation~\eqref{bp} or (if the $L^2$-scalar product is extended to a duality between $H^{-1/2}(\partial \Omega)$ and 
$H^{1/2}(\partial \Omega)$) its variational formulation
\begin{equation*}
 \scp{\op{V}\psi}{\psi'}_{L^2(\partial\Omega)} = \scp{\left(  \frac{1}{2}\op{I}+\op{K} \right)g}{\psi'}_{L^2(\partial\Omega)}, \quad \psi'
 \in H^{-1/2}(\partial \Omega),
\end{equation*}
is a consequence of the mapping properties of the single layer potential, the coercivity of the bilinear form $\scp{\op{V}\cdot}{\cdot}_{L^2(\partial\Omega)}$ and the 
Riesz-Fischer theorem.

A Galerkin approach is used in order to compute $\psi$ numerically. To this end, 
let the set $\{\psi_{1}^0,\dots, \psi_{N}^0\}$ denote the basis of the piecewise constant functions $\mathcal{P}_{0}(\mathcal{T}) \subset H^{-1/2}(\partial\Omega)$,
where $\mathcal{T}$ is a regular partition of $\partial \Omega$ into $N$ triangles. If $g$ is replaced by some piecewise linear approximation 
\begin{equation*}
 g_h \in \mathcal{P}_{1}(\mathcal{T})= \text{span} \{\psi_{1}^1,\dots, \psi_{M}^1\},
\end{equation*}
we obtain the discrete boundary integral equation $Ax = f$ with $A\in \R^{N \times N}$ and $f \in \R^N$ having the entries (see~\eqref{eq:matA})
\begin{equation*}
 \begin{split}
  a_{ij} &= \int_{\partial \Omega}\int_{\partial \Omega} K(x-y)\psi_{j}(y) \psi_{i}(x) \ud s_y \ud s_x, \quad i,j = 1,\dots,N, \\
  f_{i} &= \sum_{l = 1}^{M} g_l \scp{\left(\frac{1}{2} \op I + \op K\right) \psi_l^1}{\psi_i^0}_{L^2(\partial \Omega)}, \quad  \quad i = 1,\dots, N.
 \end{split}
\end{equation*}

\subsection*{Numerical Results} 
We choose various boundary discretizations of the ellipse $\Omega := \{ x \in \R^3: x_1^2 + x_2^2 + x_3^2/9 = 1 \}$ as the computational domain
and the Dirichlet data $ g = |x - 10 e_1|^2 $.
We compare $\H$-matrix approximations of~$A$ generated via ACA with $\H^2$-matrix approximations obtained from the method introduced in this article.
For both cases the same block cluster tree generated with $\eta = 0.8$ is used.
The minimum sizes of clusters are denoted by $\nmin$ and $\nminH$, respectively; see the remark after Definition~\ref{def:H2}.
As Table~\ref{tab:generalcompare} shows, both methods produce almost the same relative error $e_h :=\norm{u-u_h}_{L^2(\partial\Omega)} / \norm{u}_{L^2(\partial\Omega)}$,
but they differ in the time needed for computing the respective approximation of~$A$ and in the required amount of storage,
which is presented as the compression rate, i.e.\ the ratio of the amount of storage required for the approximation and the
amount of storage of the original matrix. 
\begin{table}[h!]
	\begin{center}
		\begin{tabular}{r||r|r|r|c||r|r|r|c} 
			&  \multicolumn{4}{c||}{$\H$-matrix ACA} & \multicolumn{4}{c}{$\H^2$-matrix ACA} \\
      \multicolumn{1}{c||}{$N$}     &   $\nmin$ &time in $s$   &  compr.\ in $\%$  & $e_h$      & $\nminH$  & time in $s$    &  compr.\ in $\%$   & $e_h$ \\\hline
			 $10\,024$  &  $30$     &  $0.3$  &  $16.6$  & $8.0e-4$  &  $ 400$              &  $  0.4$ & $16.0$    & $8.2e-4$ \\
			 $40\,096$  &  $60$     &  $1.5$  &  $6.1$   & $3.5e-4$  &  $ 600$              &  $  1.4$ & $ 5.5$    & $3.5e-4$ \\
			$160\,384$  &  $60$     &  $6.8$  &  $1.9$   & $2.5e-4$  &  $1000$              &  $  6.2$ & $ 1.5$    & $2.8e-4$ \\
			$641\,536$  &  $60$     & $32.2$  &  $0.6$   & $1.9e-4$  &  $1000$              &  $ 22.9$ & $ 0.4$    & $2.2e-4$
		\end{tabular}
		\caption{Comparison between $\H$- and $\H^2$-matrix adaptive cross approximation} 
		\label{tab:generalcompare}
	\end{center}
\end{table}

The time for the construction of the matrix approximation decreases the more blocks are approximated with the $\H^2$-matrix
method. While for a small number of degrees of freedom~$N$ the $\H$-matrix method is faster
than the $\H^2$-matrix method, the latter requires nearly $30\%$
less CPU time for the finest discretization.
Figures~\ref{fig:Lap_AH} and \ref{fig:Lap_AH2} give a deeper insight.
\begin{figure}[h]\centering
  \begin{minipage}[b]{.4\linewidth} 
     \includegraphics[width=\linewidth]{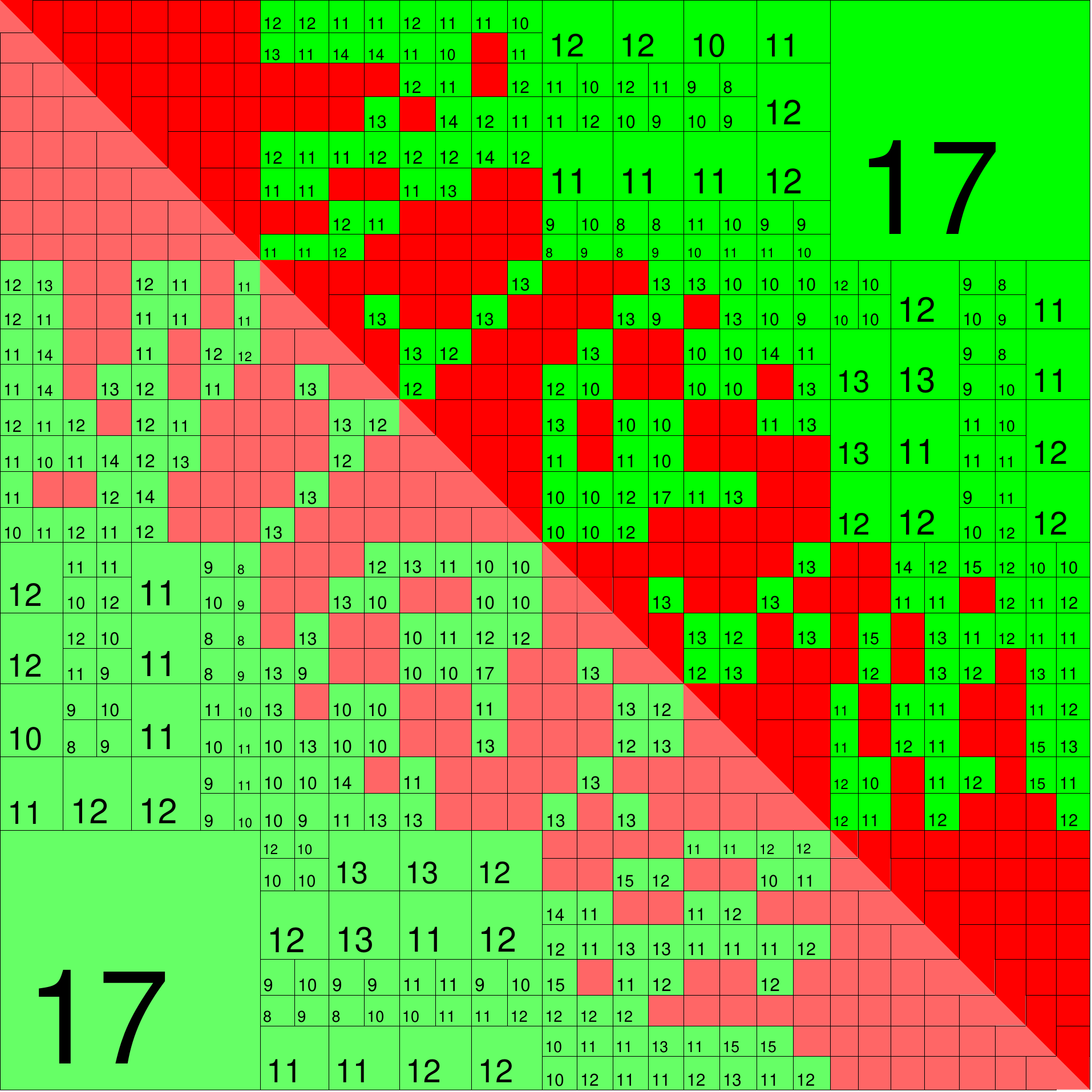}
     \caption{$A_\H$ for $N=2\,506$}
     \label{fig:Lap_AH}
  \end{minipage}
  \hspace{.1\linewidth}
  \begin{minipage}[b]{.4\linewidth} 
     \includegraphics[width=\linewidth]{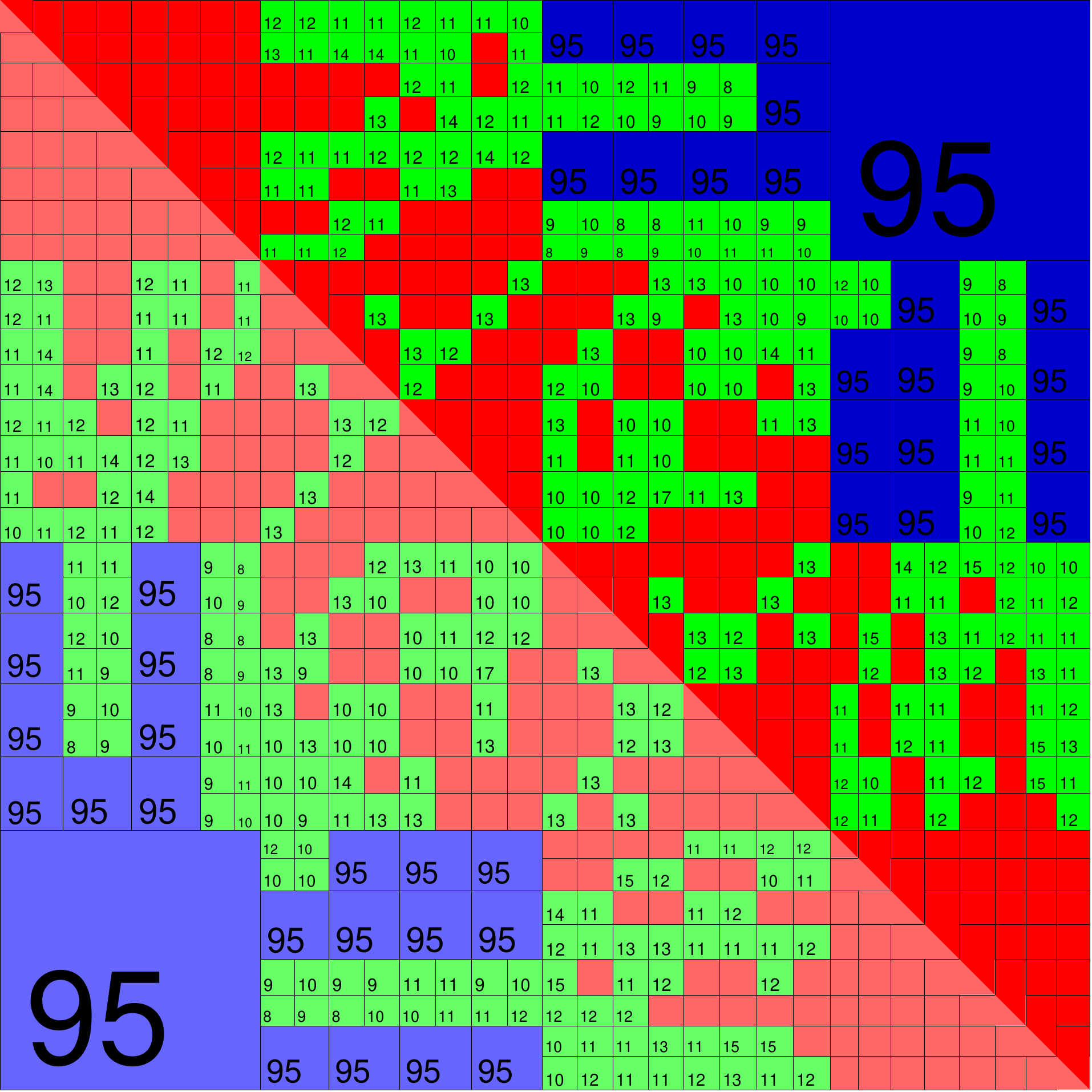}
     \caption{$A_{\H^2}$ for $N=2\,506$}
     \label{fig:Lap_AH2}
  \end{minipage} 
\end{figure}
Figure~\ref{fig:Lap_AH} shows the matrix $A$ for a coarse discretization which was approximated as an $\H$-matrix. 
Green blocks are admissible and were generated by low-rank approximation. 
The  numbers displayed in the blocks show the approximation rank~$k_\H$.
Red blocks are not admissible and were generated entry by entry.
In Figure~\ref{fig:Lap_AH2}, $A$ was approximated as an $\H^2$-matrix.
The meaning of green and red blocks is the same as in Figure~\ref{fig:Lap_AH}, the blue blocks
were generated using the $\H^2$-approximation. 
Obviously, there are several additional blocks that could be approximated with the $\H^2$-method. 
These are, however, omitted due to their size in order improve the storage requirements.

Table~\ref{tab:timecompare} shows the portion of time required for the precalculations and the time for constructing the matrix.
For the small examples the time required for the precalculation is relatively high compared to the total time 
and there are only few blocks which are approximated with the $\H^2$-method. 
Therefore the precalculations can hardly be exploited
and there is only a marginal time difference when setting up the matrices with the two methods.
The number of $\H^2$-blocks increases as the number of degrees of freedom~$N$ increases.
In this situation, the precalculations can be used more often. 
As a result, setting up the matrix with the $\H^2$-method becomes faster than with the $\H$-method.
\begin{table}[h!]
	\begin{center}
		\begin{tabular}{r||r||r|r}
                                    & \multicolumn{1}{c||}{$\H$-Matrix}  & \multicolumn{2}{c}{$\H^2$-Matrix} \\
      \multicolumn{1}{c||}{$N$}     & \multicolumn{1}{c||}{$A_{\H}$}     & \multicolumn{1}{c|}{precalculations} & \multicolumn{1}{c}{$A_{\H^2}$}    \\ \hline
			 $10\,024$  &   $0.3$ s    &   $0.1$   s  &    $0.3$ s               \\
			 $40\,096$  &   $1.5$ s    &   $0.2$   s  &    $1.2$ s               \\
			$160\,384$  &   $6.8$ s    &   $0.9$   s  &    $5.2$ s               \\
			$641\,536$  &  $32.2$ s    &   $2.6$   s  &   $20.4$ s                                      
		\end{tabular}
		\caption{Time comparison between $\H$- and $\H^2$-matrices }
		\label{tab:timecompare}
	\end{center}
\end{table}

Concerning the amount of storage, the new construction of $\H^2$-matrix approximations is
more efficient also for small numbers of degrees of freedom~$N$ as can be seen from Table~\ref{tab:memorycompare}.
\begin{table}[h]
	\begin{center}
		\begin{tabular}{r||r|r||r|r}
			&  \multicolumn{2}{c||}{$\H$-matrix} & \multicolumn{2}{c}{$\H^2$-matrix}   \\
			\multicolumn{1}{c||}{$N$}        &   memory in MB   &  compr.\ in $\%$  &  memory in MB  &  compr.\ $\%$ \\ \hline
			$10\,024$   &      $64$  &  $16.6$  &     $61$ & $16.0$     \\
			$40\,096$   &     $372$  &  $ 6.1$  &    $337$ & $ 5.5$     \\
			$160\,384$  &  $1\,824$  &  $ 1.9$  & $1\,474$ & $ 1.5$     \\
			$641\,536$  &  $8\,666$  &  $ 0.6$  & $5\,987$ & $ 0.4$    
		\end{tabular}
		\caption{Memory comparison between $\H$- and $\H^2$-matrices }
		\label{tab:memorycompare}
	\end{center}
\end{table}
The larger $N$ becomes, the more efficient is the new method.
This cannot directly be seen from the compression rates, which compare the respective
approximation with the dense matrix. However, inspecting the actual storage requirements,
one can see that the storage benefit actually improves.
For the finest discretization more than $30\%$ of storage (i.e.\ more than $2.6$~GB) are saved.

\subsection{Fractional Poisson problem}
Let $\Omega \subset \R^d$ be a Lipschitz domain, $s\in(0,1)$, and $g \in H^r(\Omega)$, $r>-s$.
We consider the fractional Poisson problem
\begin{equation}\label{eq:fracL}
\begin{aligned} 
   (-\Delta)^s u &= g \quad\text{in } \Omega, \\
               u &= 0 \quad\text{on } \R^d \backslash \Omega,
\end{aligned}
\end{equation}
where the \textit{fractional Laplacian} (see~\cite{AcBo2017}) is defined as
\begin{align*}
   (-\Delta)^s u (x) = c_{d,s}\, p.v. \int_{\R^d} \frac{u(x)-u(y)}{\lvert x-y\rvert^{d+2s}} \ud y, \quad  c_{d,s} := \frac{2^{2s}\Gamma(s+d/2)}{\pi^{d/2}\Gamma(1-s)}.
\end{align*}
Here, $s$ is called the order of the fractional Laplacian, $\Gamma$ is the Gamma function, and p.v.\ denotes the Cauchy principal value of the integral.
The solution of this problem is searched for in the Sobolev space
\[
 H^s(\Omega) = \{ v \in L^2(\Omega): |v|_{H^s(\Omega)} < \infty \},
\]
where
\begin{align*}
   |v|_{H^s(\Omega)}^2 =  \int_\Omega\int_\Omega \frac{[v(x)-v(y)]^2}{|x-y|^{d+2s}} \ud x \ud y
\end{align*}
denotes the Slobodeckij seminorm. The space $H^s(\Omega)$ is a Hilbert space, equipped with the norm
\begin{align*}
   \norm{v}_{H^s(\Omega)} = \norm{v}_{L^2(\Omega)} + |v|_{H^s(\Omega)}.
\end{align*}
Zero trace spaces $H_0^s(\Omega)$ can be defined as the closure of $C_0^\infty(\Omega)$ with respect to the $H^s$-norm.

Due to the non-local nature of the operator, we need to define the space of the test functions
\[
   \tilde H^s(\Omega) =\{ u \in L^2(\Omega):\tilde u \in H^s(\R^d)\},
\]
where  $\tilde u $ denotes the extension of $u$ by zero:
\[
  \tilde u(x) = \begin{cases}
                          u(x), & x \in \Omega, \\
                          0,    & x \in \R^d \backslash \Omega.
                        \end{cases}
\]
$\tilde H^s(\Omega)$ is also the closure of $C_0^\infty(\Omega)$ in $H^s(\R^d)$; see~\cite[Chap.~3]{McLean}.
It is known (see \cite{AiCl2018}) that $\tilde H^s(\Omega) = H_0^s(\Omega)$ for $s \neq 1/2$,
and for $s=1/2$ it holds that $\tilde H^{1/2}(\Omega) \subset H_0^{1/2}(\Omega)$.

The weak formulation of \eqref{eq:fracL} is to find $u\in \tilde H^s(\Omega)$ satisfying
\[
   a(u,v) = (g,v)_{L^2(\Omega)}, \quad v \in \tilde H^s(\Omega),
\]
where
\[
   a(u,v) =  \frac{c_{d,s}}{2} \int_\Omega \int_\Omega \frac{[u(x) - u(y)]\,[v(x)-v(y)]}{ |x-y|^{d+2s}} \ud x\ud y
           + \frac{c_{d,s}}{2s} \int_\Omega u(x)\,v(x) \int_{\partial \Omega} \frac{(y-x)^T\, n_y}{|x-y|^{d+2s}} \ud s_y \ud x.
\]
Then $\tilde H^s(\Omega)$ can be equipped with the energy norm
\[
  \lVert u \rVert_{\tilde H^s(\Omega)} =  |u|_{H^s(\R^d)} = \sqrt{a(u,u)}.
\]

Let the set $\{\varphi_{1},\dots, \varphi_{N}\}$ denote the basis of the space of piecewise linear functions~$V(\mathcal{T})$,
where $\mathcal{T}$ is a regular partition  of $\Omega$ into $M$ tetrahedra and $N$ inner points. 
The Galerkin method yields the discrete fractional Poisson problem $Ax=f$ with $A \in \R^{N \times N}$, $f \in \R^N$ having the entries
\begin{equation*}
 \begin{split}
  a_{ij} &=\quad   \frac{c_{d,s}}{2} \int_\Omega \int_\Omega \frac{[\varphi_i(x) - \varphi_i(y)]\,[\varphi_j(x)-\varphi_j(y)]}{ |x-y|^{d+2s}} \ud x\ud y \\
         &\quad  + \frac{c_{d,s}}{2s} \int_\Omega \fie_i(x)\,\fie_j(x) \int_{\partial \Omega} \frac{(y-x)^T\, n_y}{|x-y|^{d+2s}} \ud s_y \ud x. \quad i,j = 1,\dots,N, \\
  f_{i} &= (g,\fie_i)_{L^2(\Omega)}, \quad i = 1,\dots, N.
 \end{split}
\end{equation*}
If the supports of the basis functions $\fie_i$ and $\fie_j$ are disjoint, 
the computation of the entry~$a_{ij}$ simplifies to
\[
   a_{ij} = -c_{d,s} \int_\Omega \int_\Omega \frac{\fie_i(x) \fie_j(y)}{ |x-y|^{d+2s}} \ud x\ud y.
\]
Thus, admissible blocks $t\times s$ (which satisfy $\dist(X_t,X_s)>0$) are of type \eqref{eq:matA} and 
can be approximated by the method presented in this article. We remark that the
singular part $f(x,y)=|x-y|^{d+2s}$ due to its fractional exponent
is not covered by the theory of this article. Nevertheless the following numerical results show that the method works
and a theory for fractional exponents will be presented in a forthcoming article.

\subsection*{Numerical results}
The general setup and our approach is the same as in the first example in Sect.~\ref{sec:numLap}.
We compare two types of $\H$-matrix approximations of $A$ using the same block cluster tree generated with $\eta = 0.8$.
The first one is generated via ACA and the second one is an $\H^2$-matrix approximation obtained from the method introduced in this article.
Due to the Galerkin approach, we choose various volume discretizations of the ellipse $\Omega := \{ x \in \R^3: x_1^2 + x_2^2 + x_3^2/9 = 1 \}$  as the computational domain, 
the Dirichlet data $ g \equiv 1 $ and the order of the fractional Laplacian $s= 0.2$.

Since no analytical solution is known for this geometry, 
we cannot directly verify the accuracy of the numerical solution $u_h$.
Instead, we test the quality of $A_{\mathcal{H}}$ and $A_{\mathcal{H}^2}$ when applying them to a special vector.
For this purpose, we take advantage of the fact that the constant functions are in the kernel of the fractional Laplacian.
This also applies to the discrete version, the stiffness matrix~$A$.
Hence, in the following we use $e_h := \norm{ A\, \boldsymbol{1}}_2 / \sqrt{N},\, \boldsymbol{1}=[1,\dots,1]^T\in \R^N$, as
a measure of the quality of the approximations~$A_{\mathcal{H}}$ and~$A_{\mathcal{H}^2}$.

Table~\ref{tab:generalcompare_FracLap} shows the minimum sizes of the respective clusters $\nmin$ and $\nminH$ and the corresponding
numerical results, the time needed for the respective approximation of~$A$, the compression rate and the error~$e_h$.
\begin{table}[htb]
	\begin{center}
		\begin{tabular}{r||r|r|r|r||r|r|r|r} 
			          & \multicolumn{4}{c||}{$\H$-matrix ACA} & \multicolumn{4}{c}{$\H^2$-matrix ACA} \\
			\multicolumn{1}{c||}{$N$}    & $\nmin$  & time in $s$  &  compr. in $\%$  & \multicolumn{1}{c||}{$e_h$}   
			                             & $\nminH$ & time in $s$  &  compr. in $\%$  & \multicolumn{1}{c}{$e_h$} \\ \hline
			$  7\,100$  &    $30$   &      $53.4$ & $36.7$ & $2.5e-3$  &  $100$  &  $    46.6$ & $32.9$   & $2.5e-3$ \\
			$ 62\,964$  &    $60$   &  $1\,455.1$ & $11.6$ & $3.1e-4$  &  $200$  &  $1\,208.4$ & $10.1$   & $3.2e-4$ \\
			$528\,747$  &    $60$   & $28\,680.7$ & $ 2.4$ & $3.9e-5$  &  $200$  & $20\,261.9$ & $ 1.7$   & $4.4e-5$ 		
		\end{tabular}
		\caption{Comparison between $\H$- and $\H^2$-matrix adaptive cross approximation} 
		\label{tab:generalcompare_FracLap}
	\end{center}
\end{table}
As in the first example, 
the time for the construction of the matrix approximation decreases the more blocks are approximated with the $\H^2$-matrix method
and for the finest discretization the CPU time for approximating $A$ is reduced by almost $30\%$.
Here however, even for a small number of degrees of freedom $N$ the $\H^2$-method is faster.
There are two reasons for this.
The first is shown in Table \ref{tab:timecompare_FracLap}.
The cost of the precalculations is only a small fraction of the cost of the approximation of~$A$.
This is because $A$ is a dense matrix whose entries are significantly more expensive to calculate than in the first example.
\begin{table}[htb]
	\begin{center}
		\begin{tabular}{r||r||r|r}	
			                       &  \multicolumn{1}{c||}{$\H$-matrix}  & \multicolumn{2}{c}{$\H^2$-matrix} \\
		\multicolumn{1}{c||}{$N$}  &  \multicolumn{1}{c||}{$A_{\H}$}     &  precalculations &  \multicolumn{1}{c}{$A_{\H^2}$}  \\ \hline
			$  7\,100$  &      $53.4$ s & $0.4$ s &      $46.2$ s \\  
			$ 62\,964$  &  $1\,455.1$ s & $0.6$ s &  $1\,207.8$ s \\
			$526\,747$  & $28\,680.7$ s & $2.8$ s & $20\,259.1$ s   
		\end{tabular}
		\caption{Time comparison between $\mathcal{H}$- and $\mathcal{H}^2$-matrices }
		\label{tab:timecompare_FracLap}
	\end{center}
\end{table} 
The second reason can be seen from Figs.~\ref{fig:FracLap_AH} and~\ref{fig:FracLap_AH2}.
These figures show the matrix~$A$ for the coarsest discretization which was approximated as an $\H$-matrix and $\H^2$-matrix, respectively. 
As in the Figs.~\ref{fig:Lap_AH} and~\ref{fig:Lap_AH2}, the red blocks were calculated entry by entry, 
the green and blue blocks are low rank approximations calculated by the ACA and the new method, respectively, 
and the number in the low-rank blocks is the rank $k_\H$ and $k_{\H^2}$, respectively.
Compared to the first example, the ranks $k_\H$ and $k_{\H^2}$ of corresponding blocks hardly differ. 
Therefore, $\nminH$ can be chosen relatively small even for a large number of degrees of freedom $N$ in order to ensure memory efficiency 
and to approximate as many blocks as possible with the $\H^2$-method.
The reason for the small value of $k_{\H^2}$ is 
that for $|x| > 1$ the kernel function $K(x)=|x|^{-d-2s}$ is quite easy to approximate due to its decaying behavior.
For a small number of degrees of freedom $N$ the condition $|x|>1$ is almost automatically guaranteed by the admissibility condition of the $\H^2$-blocks.
On the other hand, we pay for this in the time it takes to calculate $A$, 
because the cost of the singular and near-singular integrals scale with $|\log h|$ per dimension; see \cite[Chap.~4.2]{AiCl2018}. 
\begin{figure}[htb]\centering
	\begin{minipage}[b]{.4\linewidth} 
		\includegraphics[width=\linewidth]{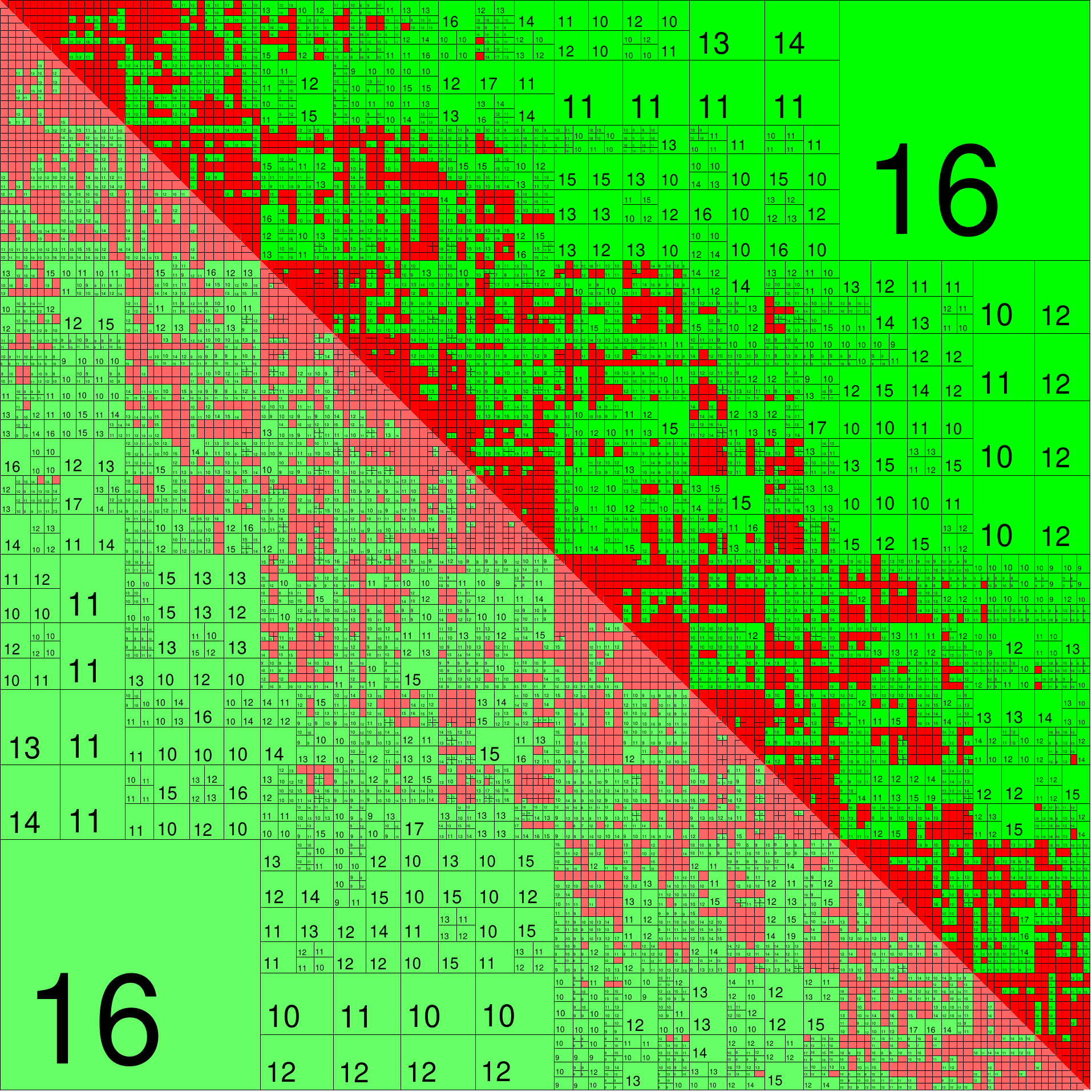}
		\caption{$A_\mathcal{H}$ for $N=7\,100$}
		\label{fig:FracLap_AH}
	\end{minipage}
	\hspace{.1\linewidth}
	\begin{minipage}[b]{.4\linewidth} 
		\includegraphics[width=\linewidth]{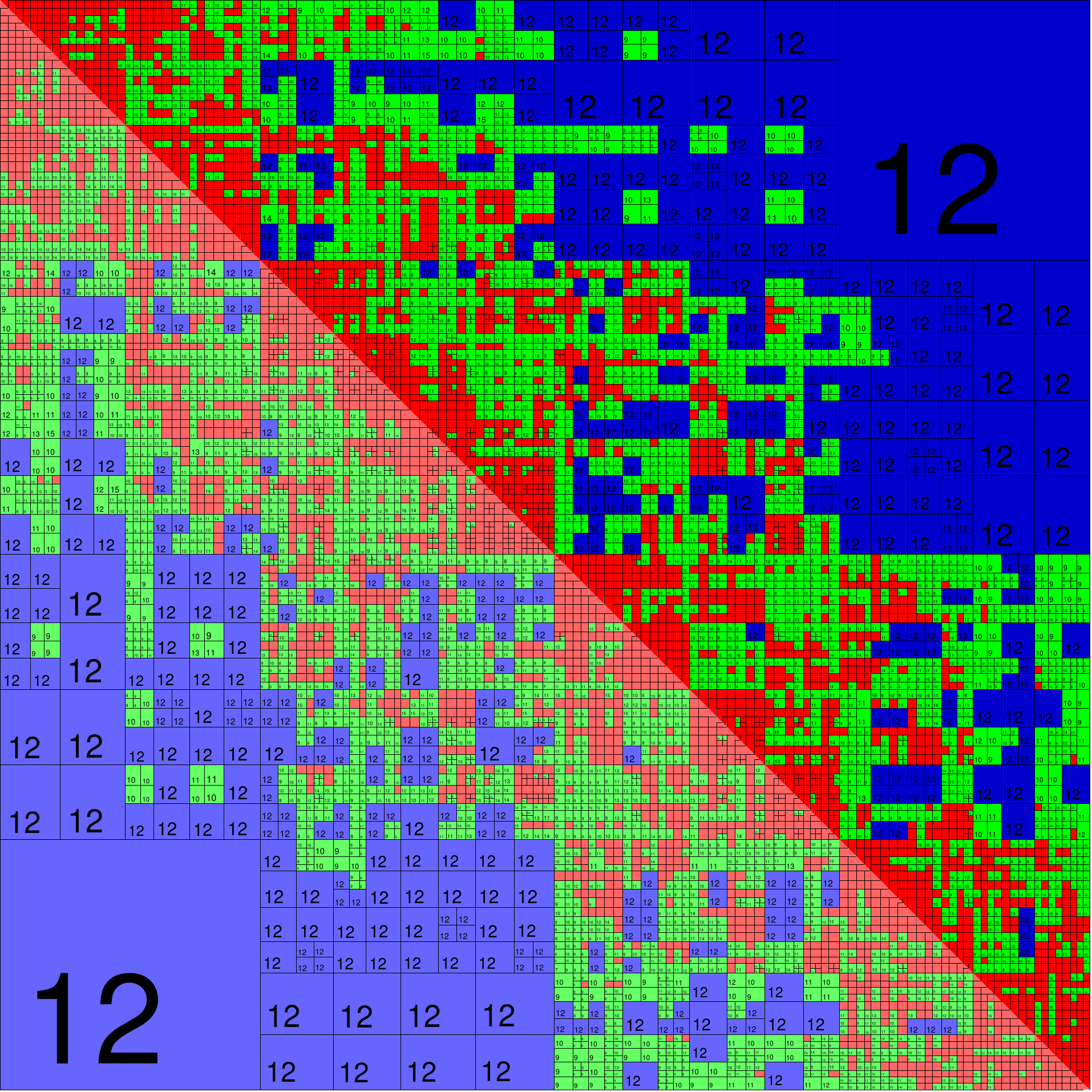}
		\caption{$A_{\mathcal{H}^2}$ for $N=7\,100$}
		\label{fig:FracLap_AH2}
	\end{minipage} 	
\end{figure}

Of course not only the CPU time benefits from the small difference between $k_\H$ and $k_{\H^2}$, but also the storage requirements as can be seen from Table~\ref{tab:memorycompare_FracLap}.
\begin{table}[h]
	\begin{center}
		\begin{tabular}{r||r|r||r|r}
			          &  \multicolumn{2}{c||}{$\H$-matrix} & \multicolumn{2}{c}{$\H^2$-matrix}  \\
			\multicolumn{1}{c||}{$N$} &  \multicolumn{1}{c|}{memory in MB} &  \multicolumn{1}{c||}{compr.\ in $\%$}
			                          &  \multicolumn{1}{c|}{memory in MB} &  \multicolumn{1}{c}{compr.\ in $\%$}   \\ \hline
			$  7\,100$  &      $71$  &  $36.7$  &      $63$ & $32.9$ \\
			$ 62\,964$  &  $1\,760$  &  $11.6$  &  $1\,527$ & $10.1$ \\
			$528\,747$  & $25\,438$  &  $ 2.4$  & $17\,993$ & $ 1.7$
		\end{tabular}
		\caption{Memory comparison between $\mathcal{H}$- and $\mathcal{H}^2$-matrices }
		\label{tab:memorycompare_FracLap}
	\end{center}
\end{table}
For each selected discretization, less storage is required when using the $\H^2$-method.
The savings are visible 
from the actual storage requirements.
For example, the finest discretization requires $30\%$ less storage (i.e.\ more than $7.4$~GB). 
In addition, the $\H^2$-approximation becomes more efficient the larger the number of degrees of freedom $N$ becomes, 
since the precalculations can be exploited for a increasingly larger part of the matrix.


\begin{thebibliography}{27}
  \bibitem{am93}
   S. Arya and  D. M. Mount.
   \newblock Approximate nearest neighbor searching.
   \newblock {\em Proc. 4th Ann. ACM-SIAM Symposium on Discrete Algorithms}, pp. 271--280, New York, ACM Press, 1993.
   
   \bibitem{am95}
   S. Arya and  D. M. Mount.
   \newblock Approximate range searching.
   \newblock {\em Proc. 11th Annual ACM Symp. on Computational Geometry}, pp. 172--181, New York, ACM Press, 1995.
   
   \bibitem{amnsw98}
   S. Arya, D. M. Mount, N. S. Netanyahu, R. Silverman, and A. Y. Wu.
   \newblock An optimal algorithm for approximate nearest neighbor searching.
   \newblock {\em J. ACM}, 45: 891--923, 1998.

   \bibitem{be54}
   H. Bateman and A. Erdélyi.
   \newblock Tables of integral transforms, Volume 2.
   \newblock Bateman Manuscript Project, McGraw-Hill, New York, USA, 1954.

   \bibitem{MR2001j:65022}
M.~Bebendorf.
\newblock Approximation of boundary element matrices.
\newblock {\em Numer. Math.}, 86(4):565--589, 2000.

   \bibitem{Bebendorf:2008}
M.~Bebendorf.
\newblock {\em {Hierarchical Matrices: A Means to Efficiently Solve Elliptic
  Boundary Value Problems}}, volume~63 of {\em Lect. Notes in Comput. Sci.
  Eng.}
\newblock Springer-Verlag, Berlin, 2008.
\newblock ISBN 978-3-540-77146-3.

\bibitem{BKV12}
M.~Bebendorf, C.~Kuske, and R.~Venn.
\newblock Wideband nested cross approximation for {H}elmholtz problems.
\newblock {\em Numer. Math.}, 130:1--34, 2015.

\bibitem{Bo10}
S.~B\"orm.
\newblock {\em Efficient Numerical Methods for non-local operators}.
\newblock Tracts in Mathematics 14. EMS, 2010.

\bibitem{BOLOME02}
S.~B{\"o}rm, M.~L\"ohndorf, and J.~M. Melenk.
\newblock Approximation of integral operators by variable-order interpolation.
\newblock {\em Numer. Math.}, 99(4):605--643, 2005.

\bibitem{BrHa09}
D.~Braess and W.~Hackbusch.
\newblock {On the efficient computation of high-dimensional integrals and the
  approximation by exponential sums}.
\newblock In Ronald~A. DeVore and Angela Kunoth, eds., {\em Multiscale,
  nonlinear and adaptive approximation}, pages 39--74. Springer, Berlin, 2009.

\bibitem{MR2000h:65178}
H.~Cheng, L.~Greengard, and V.~Rokhlin.
\newblock A fast adaptive multipole algorithm in three dimensions.
\newblock {\em J. Comput. Phys.}, 155(2):468--498, 1999.

\bibitem{ci00}
Barry~Arthur Cipra.
\newblock The best of the 20th century: Editors name top 10 algorithms.
\newblock {\em SIAM News}, 33(4), 2000.

\bibitem{MR88k:82007}
L.~F. Greengard and V.~Rokhlin.
\newblock A fast algorithm for particle simulations.
\newblock {\em J. Comput. Phys.}, 73(2):325--348, 1987.

\bibitem{MR99c:65012}
L.~F. Greengard and V.~Rokhlin.
\newblock A new version of the fast multipole method for the {L}aplace equation
  in three dimensions.
\newblock In {\em Acta numerica, 1997}, volume~6 of {\em Acta Numer.}, pages
  229--269. Cambridge Univ. Press, Cambridge, 1997.

\bibitem{MR2000c:65039}
W.~Hackbusch.
\newblock A sparse matrix arithmetic based on {$\mathcal{H}$}-matrices. {P}art
  {I}: {I}ntroduction to {$\mathcal{H}$}-matrices.
\newblock {\em Computing}, 62(2):89--108, 1999.

\bibitem{MR2001i:65053}
W.~Hackbusch and B.~N. Khoromskij.
\newblock A sparse {$\mathcal{H}$}-matrix arithmetic. {P}art {II}:
  {A}pplication to multi-dimensional problems.
\newblock {\em Computing}, 64(1):21--47, 2000.

\bibitem{WH15}
W.~Hackbusch.
\newblock {\em Hierarchical Matrices: Algorithms and Analysis}.
\newblock Springer Series in Computational Mathematics Springer Series in
  Computational Mathematics. Springer, 2015.

  \bibitem{HAKHSA00}
  W.~Hackbusch, B.~N. Khoromskij, and S.~A. Sauter.
  \newblock On $\mathcal{H}^2$-matrices.
  \newblock In H.-J. Bungartz, R.~H.~W. Hoppe, and Ch. Zenger, eds., {\em
    Lectures on Applied Mathematics}, pages 9--29. Springer-Verlag, Berlin, 2000.
  
  \bibitem{MR0100726}
  F.~Leja.
  \newblock Sur certaines suites li\'ees aux ensembles plans et leur application
    \`a la repr\'esentation conforme.
  \newblock {\em Ann. Polon. Math.}, 4:8--13, 1957.

\bibitem{mn88}
W.~R. Madych and S.~A. Nelson.
\newblock Multivariate interpolation and conditionally positive definite functions.
\newblock {\em Approx. Theory Appl.} 4, No. 4, 77--89, 1988.

\bibitem{mn90}
W.~R. Madych and S.~A. Nelson.
\newblock Multivariate interpolation and conditionally positive definite functions II.
\newblock {\em Math. Comp.}, 54:211--230, 1990.

\bibitem{mn92}
W.~R. Madych and S.~A. Nelson.
\newblock Bounds on multivariate polynomials and exponential error estimates
  for multiquadric interpolation.
\newblock {\em J. Approx. Theory}, 70:94--114, 1992.

\bibitem{McLean}
W.~McLean.
\newblock {\em Strongly Elliptic Systems and Boundary Integral Equations}.
\newblock Cambridge University Press, 2000.

\bibitem{MR86k:65120}
V.~Rokhlin.
\newblock Rapid solution of integral equations of classical potential theory.
\newblock {\em J. Comput. Phys.}, 60(2):187--207, 1985.

\bibitem{MR2054351}
Lexing Ying, George Biros, and Denis Zorin.
\newblock A kernel-independent adaptive fast multipole algorithm in two and
  three dimensions.
\newblock {\em J. Comput. Phys.}, 196(2):591--626, 2004.

\bibitem{AiCl2018}
Mark Ainsworth and Christian Clusa.
\newblock Towards an Efficient Finite ElementMethod for the Integral FractionalLaplacian on Polygonal Domains.
\newblock {\em Contemporary Computational Mathematics - A Celebration of the 80th Birthday of Ian Sloan}, pp. 17--58, Cham, Springer, 2018.

\bibitem{AcBo2017}
Gabriel  Acosta and Juan Pablo  Borthagaray 
\newblock A Fractional Laplace Equation: Regularity of Solutions and Finite Element Approximations
\newblock { \em SIAM J. Numer. Anal.}  55(2), 472--495, 2017. 
\end{thebibliography}

\end{document}